\documentclass[11pt]{amsart}

\usepackage{amsmath,amsfonts,amssymb} 
\usepackage{multirow}
\usepackage{amsthm} 
\usepackage{enumerate}
\usepackage{tikz}
\usepackage{xcolor,verbatim,colortbl}
\usepackage{graphicx,caption,subcaption} 
\usepackage{fullpage}
\usetikzlibrary{calc}

\allowdisplaybreaks 

\begingroup
    \makeatletter
    \@for\theoremstyle:=definition,remark,plain\do{%
        \expandafter\g@addto@macro\csname th@\theoremstyle\endcsname{%
            \addtolength\thm@preskip\parskip
            }%
        }
\endgroup

\newtheorem{thm}{Theorem}[section]
\newtheorem{cor}[thm]{Corollary}
\newtheorem*{thm:fl_intro}{Theorem~\ref{thm:fl_intro}}

\newtheorem{lemma}[thm]{Lemma}

\newtheorem{obs}[thm]{Observation}

\def\BB{\Lambda}
\def\wt{\widetilde}
\newcommand{\N}{\mathbb{N}}
\newcommand{\powerset}[1]{\operatorname{Pow}(#1)}

\newcommand{\cDP}{\chi_\mathrm{DP}}
\newcommand{\fDP}{f_\mathrm{DP}}

\title{\vspace{-0.5in} A lower bound on the number of edges in DP-critical graphs}

\author{Peter Bradshaw}
\address{Department of Mathematics, University of Illinois Urbana-Champaign--Champaign, Urbana, IL 61801, USA}
\email{pb38@illinois.edu}
\thanks{Peter Bradshaw received support from NSF RTG grant DMS-1937241 and an AMS Simons Travel Grant.}

\author{Ilkyoo Choi}
\address{Hankuk University of Foreign Studies, Yongin-si, Gyeonggi-do, Republic of Korea, and  Discrete Mathematics Group, Institute for Basic Science (IBS), Daejeon, Republic of Korea}
\email{ilkyoo@hufs.ac.kr}
\thanks{Ilkyoo Choi 
is supported by the Hankuk University of Foreign Studies Research Fund, the Institute for Basic Science (IBS-R029-C1), and the National Research Foundation of Korea(NRF) grant funded by the Korea government(MSIT) (RS-2025-23324220).}

\author{Alexandr Kostochka}
\address{Department of Mathematics, University of Illinois Urbana-Champaign--Champaign, Urbana, IL 61801, USA}
\email{kostochk@illinois.edu}
\thanks{Alexandr Kostochka 
is supported in part by  NSF  Grant DMS-2153507 and by NSF RTG Grant DMS-1937241.}

\author{Jingwei Xu}
\address{Department of Mathematics, University of Illinois Urbana-Champaign--Champaign, Urbana, IL 61801, USA}
\email{jx6@illinois.edu}
\thanks{Jingwei Xu
is supported in part by   Campus Research Board Award RB24000 of the University of Illinois Urbana-Champaign.}

\begin{document}
\maketitle

\vspace{-0.3in}

\begin{abstract}
A graph $G$ is $k$-critical (list $k$-critical,  DP $k$-critical) if $\chi(G)= k$ ($\chi_\ell(G)= k$, $\cDP(G)= k$) and for every proper subgraph $G'$ of $G$, $\chi(G')<k$ ($\chi_\ell(G')< k$, $\cDP(G')<k$). 
Let $f(n, k)$ ($f_\ell(n, k), \fDP(n,k)$) denote the minimum number of edges in an $n$-vertex
$k$-critical (list $k$-critical, DP $k$-critical) graph. 
Our main result is that if $k\geq 5$ and $n\geq k+2$, then
$$\fDP(n,k)>\left(k - 1 + \left \lceil \frac{k^2 - 7}{2k-7}  \right \rceil^{-1}\right)\frac{n}{2}.
$$
This is the first bound on $\fDP(n,k)$ that is asymptotically better than the well-known bound on $f(n,k)$ by Gallai from 1963.
The result also yields a slightly better bound on $f_{\ell}(n,k)$ than the ones known before.
\vspace{0.15in}
 
\end{abstract}

 \section{Introduction}
In this paper, a {\em graph} has neither parallel edges nor loops. 
A {\em multigraph} may have parallel edges, but no loops. 
Given a multigraph $G$, let $V(G)$ and $E(G)$ denote the vertex set and edge set, respectively, of $G$. 
Given a vertex $v\in V(G)$, the {\em degree} of $v$, denoted $d_G(v)$, is the number of edges incident with $v$. 
The {\em neighborhood} $N_G(v)$ of a vertex $v$ is the set of vertices adjacent to $v$.
Note that $|N_G(v)| \leq d_G(v)$, and equality holds if and only if $v$ has no incident parallel edges.
A vertex of degree $d$ (at least $d$) is a {\em $d$-vertex} ({\em $d^+$-vertex}). 
For vertex subsets $S_1$ and $S_2$, let $E_G(S_1, S_2)$ denote the set of edges $xy \in E(G)$ where $x\in S_1$ and $y\in S_2$. For integers $a, b$ with $a\leq b$, let $[a,b]$ denote the set $\{a, a+1, \ldots, b\}$ of integers.

 \subsection{Known results on proper coloring}

A {\em (proper) $k$-coloring} of a graph $G$ is a mapping $\varphi\,: \,V(G)\to [1,k]$ such that $\varphi(u)\neq \varphi(v)$ for each $uv\in E(G)$.
The {\em chromatic number} of $G$,  denoted $\chi(G)$, is the minimum positive integer $k$ for which $G$ has a proper $k$-coloring.
A graph $G$ is {\em $k$-colorable} if $\chi(G)\leq k$. 
 For a positive integer $k$, a graph $G$ is {\em $k$-critical} if $\chi(G)=k$, but every
proper subgraph of $G$ is $(k-1)$-colorable.

The notion of $k$-critical graphs was introduced and systematically studied by
Dirac~\cite{1951Dirac
 ,1952DiracSome,1953Dirac,1957DiracMap,1957DiracAtheorem,
1974Dirac} starting from 1951.  
In particular, Dirac considered the minimum number $f(n,k)$ of edges in an $n$-vertex $k$-critical graph.
With this notation, $f(k,k)=\binom{k}{2}$
and $f(k+1,k)$ is not well defined because there is no $(k+1)$-vertex $k$-critical graph. 
Since every vertex of a $k$-critical graph has degree at least $k-1$, we trivially have $f(n,k)\geq (k-1)\frac{n}{2}$. Even this
 simple bound  yields the Heawood Formula~\cite{Hea}
that every graph $G$ embeddable in an orientable surface $S_{\gamma}$ with
$\gamma\geq 1$ handles satisfies $\chi(G)\leq \lfloor c_\gamma\rfloor$, where 
$c_\gamma= \frac{7+\sqrt{1+48\gamma}}{2}$. The well-known Brooks' Theorem~\cite{Bro} implies that for $k\geq 4$ and $n\geq k+2$, $f(n,k)\geq \frac{(k-1)n+1}{2}$.
Dirac sharpened Heawood's result by showing that for $\gamma\geq 1$, every graph embeddable in $S_{\gamma}$ with chromatic number $\lfloor c_{\gamma}\rfloor$ contains the complete graph with $\lfloor c_{\gamma}\rfloor$ vertices as a subgraph. 
 For this, he used the following lower bound on $f(n,k)$:

\begin{thm}[Dirac~\cite{1957DiracAtheorem}]\label{D2}
If $k\geq 4$ and $n \geq k+2$, 
then 
\begin{equation}\label{Di1}
f(n,k)\geq \left(k-1\right)\frac{n}{2}+\frac{k-3}{2}.
 \end{equation}
\end{thm}

Fifteen years later, Dirac~\cite{1974Dirac} described the $k$-critical graphs for which the bound~\eqref{Di1} is exact.

In his fundamental papers~\cite{1963Gallai1} and~\cite{1963Gallai2} from 1963, Gallai proved a series of important properties of color-critical graphs. In particular, he found the exact values of $f(n,k)$ for
all $n\in[k+2, 2k-1]$.

Recall that
 a {\em block} in a graph $G$ is a maximal 
connected subgraph of $G$ that has no cut vertices. 
A {\em Gallai tree} is a connected graph whose every block is a complete graph or an odd cycle, and a~\emph{Gallai forest} is a graph in which every connected component is a Gallai tree.

\begin{thm}[Gallai~\cite{1963Gallai2}]\label{Ga2}
Let $k\geq 4$, and let $G$ be a $k$-critical graph. If $B$ is the set of $(k-1)$-vertices in $G$,
then $G[B]$ is a Gallai forest.
\end{thm}
Gallai~\cite{1963Gallai1} showed that his 
Theorem~\ref{Ga2}  implies the following lower bound on $f(n,k)$:
if $k \geq 4$ and $n\geq k+2$, then
\begin{equation}\label{in3}
f(n,k)\geq \left({k-1}+\frac{k-3}{k^2-3}\right)\frac{n}{2}.
\end{equation}
For large $n$, this bound is much stronger than~\eqref{Di1}.

Krivelevich~\cite{1997Krivelevich,1998Krivelevich} improved the lower bound on $f(n, k)$ in~\eqref{in3} 
 to
\begin{equation}\label{in5}
f(n,k)\geq \left(k-1+\frac{k-3}{k^2-2k-1}\right)\frac{n}{2}
\end{equation}
and demonstrated nice applications of his bound.
Then, Kostochka and Stiebitz~\cite{2003KS} proved that for $k\geq 6$ and $n\geq k+2$,
\begin{equation}\label{in6}
f(n,k)\geq \left(k-1+\frac{2(k-3)}{k^2+6k-9-6/(k-2)}\right)\frac{n}{2}.
\end{equation}
Note that the last term in the bound~\eqref{in6} is asymptotically (in $k$) twice larger than the one in~\eqref{in3}. Then, Kostochka and Yancey~\cite{2014KoYa} proved an asymptotically exact bound:
if  $k\geq 4$ and $n\geq k,\,n\neq k+1$, then
\begin{equation}\label{KYa}
    f(n,k)\geq   \left(k-1+\frac{k-3}{k-1}\right)\frac{n}{2}-
\frac{k(k-3)}{2(k-1)}.
\end{equation}
This bound on $f(n,4)$ yields a short proof of the well-known theorem by Gr\" otzsch~\cite{Gro} that every triangle-free planar graph is $3$-colorable.

 \subsection{Known results on list coloring}

\emph{List coloring} was introduced independently by Vizing~\cite{1976Vizing} and Erd\H os, Rubin, and Taylor~\cite{1980ErRuTa}. 
For a set $X$, let $\powerset{X}$ denote the power set of a set $X$, and denote $\bigcup_{v\in X}f(v)$ by $f(X)$.
A \emph{list assignment} for a graph $G$ is a function $L \colon V(G) \to \powerset{Y}$, where $Y$ is a set, whose elements are referred to as \emph{colors}. 
For each $u \in V(G)$, the set $L(u)$ is called the \emph{list} of $u$ and its elements are said to be \emph{available} for~$u$. 
A proper coloring $\varphi \colon V(G) \to Y$ is called an \emph{$L$-coloring} if  $\varphi(u) \in L(u)$ for each $u \in V(G)$. 
A graph $G$ with a list assignment $L$ is said to be \emph{$L$-colorable} if it admits an $L$-coloring. 
The \emph{list chromatic number} $\chi_\ell(G)$ of $G$  is the least positive integer $k$ such that $G$ is $L$-colorable whenever $L$ is a list assignment for $G$ with $|L(u)| \geq k$ for all $u \in V(G)$. 
If $L(u) = [1,k]$ for a positive integer $k$ for all $u \in V(G)$, then $G$ is $L$-colorable if and only if it is $k$-colorable; in this sense, list coloring generalizes proper coloring. In particular, $\chi_\ell(G) \geq \chi(G)$ for all graphs $G$.

The definition of critical graphs can be naturally extended to the list coloring setting. 
A graph $G$ is {\em list $k$-critical} if $\chi_\ell(G) = k$, but
$\chi_\ell(G') \leq k-1$ for each proper subgraph $G'$ of $G$.
So, we can define $f_\ell(n,k)$  to be the minimum number of edges in an $n$-vertex list $k$-critical graph.

A list assignment $L$ for a graph $G$ is called a \emph{degree-list assignment} if $|L(u)| \geq d_G(u)$ for all $u \in V(G)$. A fundamental result of Borodin~\cite{1979Borodin} and Erd\H os, Rubin, and Taylor~\cite{1980ErRuTa} provides a complete characterization of all graphs that are not $L$-colorable with respect to some degree-list assignment $L$. 
This result can be viewed as a generalization of Theorem~\ref{Ga2}.
	
\begin{thm}[Borodin~\cite{1979Borodin}; Erd\H os, Rubin, and Taylor~\cite{1980ErRuTa}]\label{theo:list_Brooks}
Let $G$ be a connected graph and $L$ be a degree-list assignment for $G$. If $G$ is not $L$-colorable, then $G$ is a Gallai tree; furthermore, $|L(u)| = d_G(u)$ for all $u \in V(G)$ and if $u, v \in V(G)$ are two adjacent non-cut vertices, then $L(u) = L(v)$.
	\end{thm}

Kostochka, Stiebitz, and Wirth~\cite{KSW} and independently Thomassen~\cite{Tho} described for each Gallai tree $G$ the degree-list assignments $L$    such that $G$ is not $L$-colorable.
 Kostochka and  Stiebitz~\cite{KS02} proved that for $n \geq k+2$, Dirac’s bound (1) holds  for $f_\ell(n, k)$. 

Exactly as Theorem~\ref{Ga2} implies~\eqref{in3}, Theorem~\ref{theo:list_Brooks} yields the same lower bound on~$f_\ell(n,k)$. 
Kostochka and Stiebitz~\cite{2003KS} improved~\eqref{in3} for $k\geq 9$: the last term in their bound is asymptotically (in $k$) $1.5$ times larger than the one in~\eqref{in3}.
Then, in a series of papers by Kierstead and Rabern~\cite{2020KiRa}, 
Cranston and Rabern~\cite{2018CrRa}, and Rabern~\cite{2016Rabern,2018Rabern}, 
this bound was significantly improved and also extended to $4\leq k\leq 8$. 
The best known bounds are due to Rabern~\cite{2016Rabern,2018Rabern}:

\begin{thm}[Rabern~\cite{2016Rabern,2018Rabern}]
\label{Cr-Ra}
For $k\geq 4$ and $n\geq k+2$, 
\[\displaystyle f_\ell(n,k)\geq
\begin{cases}
     \displaystyle\left(k-1+\frac{(k-3)^2(2k-3)}{k^4-2k^3-11k^2+28k-14}\right)\frac{n}{2} & \mbox{ if } k\geq 7\\[2ex]
    \displaystyle\left(5+\frac{93}{766}\right)\frac{n}{2}& \mbox{ if } k=6\\[2ex]
    \displaystyle\left(k-1+\frac{k-3}{k^2-2k+2}\right)\frac{n}{2}& \mbox{ if } k\in\{4,5\}\\
\end{cases}\]
\end{thm}
In fact,  the bounds were proved not only for $ f_\ell(n,k)$ but also for the analogous parameter of paintability,
which is defined in terms of an on-line list coloring process where available colors are revealed one at a time (see~\cite{2016Rabern,2018Rabern} for a complete definition).
The last term in Theorem~\ref{Cr-Ra} is
asymptotically (in $k$) twice 
larger than the one in~\eqref{in3}.
Below is a table similar to ones in~\cite{2018CrRa,2020KiRa,2016Rabern} with an added column of our results in this paper.

\begin{table}[h!]
\centering
\begin{tabular}{|r||c|c||c|c|c|c||c|} 
 \hline
 & \multicolumn{2}{c||}{$k$-critical} 
 & \multicolumn{4}{c||}{list $k$-critical} & 
  list and DP $k$-critical  \\
 \hline
$k$ &Ga~\cite{1963Gallai1}  & KY~\cite{2014KoYa} & KS~\cite{2003KS} & KR~\cite{2020KiRa} &CR~\cite{2018CrRa} & Ra~\cite{2016Rabern,2018Rabern} & This paper \\ 
 \hline\hline
4 & \phantom{0}3.0769 & \phantom{0}3.3333 &   &   &   & \phantom{0}3.1000$^*$ &   \\ 
5 & \phantom{0}4.0909 & \phantom{0}4.5000 &   & \phantom{0}4.0983 & \phantom{0}4.1000 & \phantom{0}4.1176 & \phantom{0}4.1666$^*$ \\ 
6 & \phantom{0}5.0909 & \phantom{0}5.6000 &   & \phantom{0}5.1052 & \phantom{0}5.1076 & \phantom{0}5.1214 & \phantom{0}5.1666$^*$ \\ 
7 & \phantom{0}6.0869 & \phantom{0}6.6666 &   & \phantom{0}6.1149 & \phantom{0}6.1192 & \phantom{0}6.1296 & \phantom{0}6.1666$^*$ \\ 
8 & \phantom{0}7.0819 & \phantom{0}7.7142 &   & \phantom{0}7.1127 & \phantom{0}7.1167 & \phantom{0}7.1260 & \phantom{0}7.1428$^*$ \\ 
9 & \phantom{0}8.0769 & \phantom{0}8.7500 & \phantom{0}8.0838 & \phantom{0}8.1093 & \phantom{0}8.1130 & \phantom{0}8.1213 & \phantom{0}8.1428$^*$ \\ 
10 & \phantom{0}9.0721 & \phantom{0}9.7777 & \phantom{0}9.0793 & \phantom{0}9.1054 & \phantom{0}9.1088 & \phantom{0}9.1162 & \phantom{0}9.1250$^*$ \\ 
15 & 14.0540 & 14.8571 & 14.0610 & 14.0863 & 14.0884 & 14.0930 & 14.1000$^*$ \\ 
20 & 19.0428 & 19.8947 & 19.0490 & 19.0718 & 19.0733 & 19.0762 & 19.0833$^*$ \\ 
\hline 
\end{tabular}
\caption{Current lower bounds on the average degree of $k$-critical and list $k$-critical graphs. 
Best results for list $k$-critical graphs are marked with an asterisk.}
\label{table:1}
\end{table}

\subsection{Known results on DP-coloring and our results}

In this paper we focus on a generalization of list coloring that was recently introduced by Dvo\v r\' ak and Postle~\cite{2018DvPo}; 
they called it \emph{correspondence   coloring}, and we call it \emph{DP-coloring} for short. 
Dvo\v r\' ak and Postle invented DP-coloring in order to approach an open problem about list colorings of planar graphs with no cycles of certain lengths.
This hints at the usefulness of DP-coloring for graph coloring problems, in particular list coloring problems. 

For a multigraph $G$, a \emph{(DP-)cover} of $G$ is a pair $(H, L)$, where $H$ is a multigraph and $L \colon V(G) \to \powerset{V(H)}$ is a function such that 
\begin{itemize}
    \item The family $\{L(u) : u \in V(G)\}$ forms a partition of $V(H)$. 
    \item For each $u \in V(G)$, $H[L(u)]$ is an independent set.
    \item For each $u,v \in V(G)$, if $|E_G(u,v)|=s$, then $E_H(L(u), L(v))$ is the union of $s$ matchings (where each matching is not necessarily perfect and possibly empty). 
\end{itemize}
An $(H,L)$\emph{-coloring} of $G$ is a function
$\phi:V(G)\rightarrow \bigcup_{v\in V(G)}L(v)$
with $\phi(v)\in L(v)$ such that $\{(v,\phi(v)):v\in V(G)\}$ is independent in $H$. In view of this definition, we often refer to the vertices of $H$ as \emph{colors}.
The \emph{DP chromatic number} $\cDP(G)$ of a multigraph $G$  is the least positive integer $k$ such that $G$ has an $(H,L)$-coloring
whenever $(H,L)$ is a cover of $G$ with $|L(u)| \geq k$ for all $u \in V(G)$.
Every list coloring problem can be represented as a DP-coloring problem.
In particular, $\cDP(G) \geq \chi_\ell(G)$ for all multigraphs $G$.

A cover $(H, L)$ of a multigraph $G$ is called a \emph{degree-cover} if $|L(u)| \geq d_G(u)$ for all $u \in V(G)$. 
We say a multigraph $G$ is \emph{DP degree-colorable} if $G$ has an $(H, L)$-coloring whenever $(H, L)$ is a degree-cover.

A multigraph $G$ is \emph{DP $k$-critical} if $\cDP(G) = k$ and $\cDP(G') \leq k-1$ for every proper subgraph $G'$ of $G$.
Let $\fDP(n,k)$ denote the minimum number of edges in an $n$-vertex DP $k$-critical  graph.

For a graph $G$ and a positive integer $s$, the {\em multiple} $s G$ of $G$ is the multigraph obtained from $G$ by replacing each edge  $e \in E(G)$
with $s$  edges joining the endpoints of $e$. 
In particular, $1 G=G$.
A \emph{GDP-forest} is a multigraph such that for every block $B$, there exist $n$ and $t$ such that
$B$ is either a $t K_n$ or a $t C_n$.
 (A {\em double cycle} is a multigraph $2 C_n$.)
A {\em GDP-tree} is a connected GDP-forest. 
Note that every Gallai tree is also a GDP-tree.
An example of a GDP-tree is shown in Figure~\ref{fig:Gallai-tree}.

\begin{figure}
\begin{center}
\begin{tikzpicture}
[scale=1.2,auto=left,every node/.style={circle,fill=gray!30,minimum size = 6pt,inner sep=0pt}]

\node(v4) at (0,1) [draw = black] {};
\node(v1) at (0,1.8) [draw = black] {};
\node(v2) at (-0.4,1.4) [draw = black] {};
\node(v3) at (0.4,1.4) [draw = black] {};
\node(p1) at ($(-0.4,1.4) + (-1,0.577) $)[draw = black] {};
\node(p3) at ($(-0.4,1.4) + (-1,-0.577) $)[draw = black] {};
\node(w) at (0,2.6) [draw = black] {};
\foreach \i in {1,...,7}
    \node[draw = black] (q\i) at ($ (1.4-0.4,1.4) + ({180+360/8*\i}:0.6) $)  {};
\node(r) at ($(p1) + (-0.8,0)$) [draw=black] {};
\node(x1) at ($(r) + (-0.6,0.2)$) [draw = black] {};
\node(x2) at ($(r) + (-0.6,-0.2)$) [draw = black] {};

\node(y1) at ($(p3) + (-0.6,0.2)$) [draw = black] {};
\node(y2) at ($(p3) + (-0.6,-0.2)$) [draw = black] {};
\foreach \from/\to in {r/p1,p1/v2,p1/p3,v2/p3,v3/q1,q1/q2,q2/q3,q3/q4,q4/q5,q5/q6,q6/q7,q7/v3,v1/v2,v2/v3,v3/v1,v1/v4,v2/v4,v3/v4,w/v1,r/x1,r/x2,x1/x2,y1/y2,y1/p3,y2/p3}
    \draw (\from) -- (\to);

\draw[bend left=15] (p1) to (p3);
\draw[bend left=15] (p3) to (v2);
\draw[bend right=15] (p1) to (v2);
\end{tikzpicture}
\end{center}
\caption{A GDP-tree with two $K_3$ blocks, a $2K_3$ block, two $K_2$ blocks, a $K_4$ block, and a $C_8$ block.}
\label{fig:Gallai-tree}
\end{figure}

Theorems~\ref{Ga2} and~\ref{theo:list_Brooks} extend  to DP-coloring as follows. 

\begin{thm}[Dvo\v r\' ak and Postle~\cite{2018DvPo} for graphs; Bernshteyn, Kostochka, and Pron’~\cite{2017BeKoPr} for multigraphs]
\label{thm:deg-choosable}
Let $G$ be a connected multigraph, and let $(H,L)$ be a degree-cover of $G$. 
If $G$ has no $(H,L)$-coloring, then $G$ is a GDP-tree.
\end{thm}

Kim and Ozeki~\cite{KimO} characterized for each GDP-tree the degree-covers $(H,L)$ 
for which $G$ is not $(H, L)$-colorable.

Analogous to how Theorems~\ref{Ga2} and~\ref{theo:list_Brooks} imply lower bounds on $f(n, k)$ and $f_\ell(n, k)$, respectively, Theorem~\ref{thm:deg-choosable} yields that the lower bound in~\eqref{in3} holds also for $\fDP(n,k)$ when $n\geq k+2$. 
Note that the bound~\eqref{in3} does not hold for DP-coloring of general multigraphs: for example, multiples of cycles
$k C_n$ are $2k$-regular and DP $(2k+1)$-critical.


 Bernshteyn and Kostochka~\cite{2018BeKo} proved that for $n\geq k+2$,
Dirac's bound~\eqref{Di1} 
 holds also for $\fDP(n,k)$.
The reader can find more history of the study of color-critical graphs, including seminal results by Gallai, Haj\' os, Sachs, and Stiebitz  in the recent monographs~\cite{Cran,2024StScTo}.

The main result of this paper is a lower bound on $\fDP(n,k)$
that is asymptotically better than~\eqref{in3}. 

\begin{thm}\label{thethm} 
 If $k \geq 5$,  $\lambda  = \left \lceil \frac{k^2 - 7}{2k-7}\right \rceil$, and $n \geq k+2$,     then 
    \begin{equation}\label{eqcor62'}
        \fDP(n,k) \geq \left(k-1 + \frac {1}{\lambda} \right) \frac n2 + \frac 1 {\lambda}.        
    \end{equation}    
\end{thm}

We do not know whether $f_{\ell}(n,k)\geq \fDP(n,k)$ for all $n$ and $k$.
However, our method for proving Theorem~\ref{thethm} also allows us to show the following bound on $f_{\ell}(n,k)$.

\begin{thm}\label{thm:fl_intro}
    If $k \geq 5$,  $\lambda  = \left \lceil \frac{k^2 - 7}{2k-7}\right \rceil$, and $n \geq k+2$,     then 
    \begin{equation}\label{eqcor62}
        f_{\ell}(n,k) \geq \left(k-1 + \frac {1}{\lambda} \right) \frac n2 + \frac 1 {\lambda}.        
    \end{equation}
\end{thm}

This is slightly better than the bound of Theorem~\ref{Cr-Ra} for $k\geq 5$, see the last column of Table~1.
We do not know how to prove~\eqref{eqcor62} without using DP-coloring.
Thus, this is another example of usefulness of DP-coloring for proving results on list coloring.
As it is, an analog of Theorem~\ref{thethm} for $k=4$ requires additional ideas to work. In~\cite{k=4}, we use a somewhat more elaborate technique to prove that for $n\geq 11$, $\fDP(n,4) > \frac{8}{5}n$, which is better than~\eqref{eqcor62'} for $k=4$.

\subsection{Outline of the proof and structure of the paper}

The plan to prove Theorem~\ref{thethm} is as follows.

\begin{enumerate}
    \item We want to prove by induction on $n$ that if $G\neq K_k$ is an $n$-vertex DP $k$-critical graph with $m$ edges, then $2m\geq \left(k - 1 + \frac 1\lambda\right)n+\frac{2}{\lambda}$, which is equivalent to
    \begin{equation}\label{out1}
      ( \lambda (k-1)+1)n-2\lambda m\leq -2.
    \end{equation}
 \item  We wish to consider a smallest counterexample $G$ to~\eqref{out1} and 
 use the minimality of $G$ to
 consider smaller graphs for which~\eqref{out1} holds.  In doing so, we hope to
 infer some structural properties of $G$. 
  For this, we use a weighted version of~\eqref{out1}. If we assign
 $w(v)= \lambda (k-1)+1$ for each $v\in V(G)$ and $w(e)=2\lambda$ for each $e\in E(G)$, then we can rewrite~\eqref{out1} as
  \begin{equation}\label{out2}
      \sum_{v\in V(G)}\left( w(v)-\sum_{e\in E(G): v\sim e}\frac{w(e)}{2}\right) \leq -2.
    \end{equation}
   Note that vertices of degree at most $k-1$ contribute positive terms to the sum, while vertices of degree at least $k$ contribute negative terms. 
   Therefore, in order to prove the inequality, 
   we plan to prove that the $k^+$-vertices of $G$ contribute enough to the sum in order to negate the positive terms from the vertices of degree at most $k-1$.
   
 \item Sometimes it is convenient to consider smaller graphs obtained from $G$ in which some vertices have lists with fewer than $k-1$ colors.
 For induction, it is also convenient to consider multigraphs rather than just graphs.
 Therefore,
 we develop a more general and sophisticated model in which the weights of  vertices with smaller list sizes are less than 
$ \lambda (k-1)+1$ and the weights of multiple edges are greater than
$2\lambda$. In this new model, we still aim to prove~\eqref{out2}.

\item In our new model, every positive term of~\eqref{out2} corresponds to a {\em low} vertex, that is, a vertex $v$ for which the list size $h(v)$ is equal to the degree $d_G(v)$.
We let $\BB$ be the set of low vertices in $G$, 
a minimum counterexample to our new model.
We show that the induced subgraph $G[\BB]$ is a GDP-forest.

\item Overall, roughly speaking, the proof of the weighted version of~\eqref{out2} has three parts:

(a) We show that GDP-forests of maximum degree $k-1$ and without $(k-1)$-regular and $(k-2)$-regular blocks do not have too many edges.

(b) We show that in a minimum counterexample $G$, $G[\BB]$ has neither $(k-1)$-regular nor $(k-2)$-regular blocks.

(c) We use discharging to show that negative terms in~\eqref{out2} compensate for all positive ones, proving the inequality.
    
\end{enumerate}

The structure of the paper is as follows.
In Section~\ref{sec:setup}, we introduce our model discussed above, which allows multiple edges and variable list sizes of the vertices.
In Section~\ref{sec:prelim} we discuss properties of GDP-forests.
Properties of a minimum counterexample are shown in Section~\ref{sec:propG}, and we finish the proof via the discharging method in Section~\ref{sec:discharging}. 
In Section~\ref{listco} we prove the bound for $f_\ell(n,k)$.
In conclusion, we briefly discuss possible further research directions.

\section{The setup}\label{sec:setup}

We will prove a  statement stronger than Theorem~\ref{thethm} in the setting of multigraphs with the language of potentials, and we will obtain Theorem~\ref{thethm} as a corollary.

Let $G$ be a multigraph,
and let $h:V(G) \to [0,k-1]$ be a function. 
For each $v,x,y \in V(G)$, 
we define the \emph{potential} $\rho_{G,h}(v)$ and $\rho_{G,h}(xy)$ to be as follows:
\[
\hfill
\rho_{G,h}(v)=
\begin{cases}
h(v)\lambda +1 & \textrm{if } h(v) = k-1 \\
h(v)\lambda -1 & \textrm{if } h(v)\in [2,k-2] \\
h(v)\lambda -2 & \textrm{if } h(v) \in \{0,1\}.
\end {cases}
\phantom{0}\hfill\phantom{0}
\rho_{G,h}(xy) = 
\begin{cases}
0 & \textrm{if } |E_G(x,y)| = 0 \\
1 - (2\lambda +1)|E_G(x,y)| 
& \textrm{otherwise.}
\end{cases}
\hfill
\]
In other words, if $x$ and $y$ are joined by a single edge in $G$, then $\rho_{G,h}(xy) = -2\lambda $, and each additional edge joining $x$ and $y$ adds $-(2\lambda +1)$ to $\rho_{G,h}(xy)$.

Given a vertex subset $A \subseteq V(G)$, 
we define the \emph{potential} of $A$ as 
\[\rho_{G,h}(A) = \sum_{x \in V(A) } \rho_{G,h}(x) + \sum_{xy \in  \binom{A}{2} } \rho_{G,h}(xy) .\]
When $A$ consists of a single vertex $v$, we often write $ \rho_h(v) = \rho_{G,h}(\{v\})$.
We also write $\rho_h(G) = \rho_{G,h}(V(G))$. 
We often omit the subscripts $G$ and $h$ when they are clear from the context.

Let $\N$ denote the set of positive integers.
Given a multigraph $G$, let $h:V(G) \to \N \cup\{0\}$.
A cover $(H,L)$ of $G$ is an {\em $h$-cover} of $G$ if $|L(v)| \geq h(v)$ for each $v \in V(G)$.
When $h\equiv t$ for some constant $t\in \N$, we call an $h$-cover simply a {\em $t$-cover}.
We say that $G$ is {\em DP $h$-colorable} if $G$ has an $(H,L)$-coloring for every $h$-cover $(H,L)$ of $G$. 
We say that $G$ is (DP) {\em $h$-minimal}
if $G$ is not DP $h$-colorable, but every proper subgraph $G'$ of $G$ is DP $h\vert_{V(G')}$-colorable.
Given a vertex subset $A \subseteq V(G)$, we often write $h$ for the restriction $h\vert_{A}$ when this does not cause confusion.

We now present our main result:

\begin{thm}
\label{thm:stronger}
    Let $k \geq 5$.
    Let $G$ be a multigraph
    and  $h:V(G) \to [0,k-1]$.
     If $G$ is $h$-minimal, then one of the following holds:
    \begin{enumerate}
        \item $G=K_{k}$ and $h(v) \in \{k-1, k-2\}$ for each $v \in V(G)$,
        with at most one vertex $w \in V(G)$ satisfying $h(w) = k-2$, 
        
        \item $k=5$, either $G =4 K_2$ or $G$ is a double cycle, and $h(v) = 4$ for each $v \in V(G)$, or
        \item $\rho_h(G) \leq -2$.
    \end{enumerate}
\end{thm}

Note that  when $G$ is a graph and $h(v) = k-1$ for each $v \in V(G)$,
the statement $\rho_h(G) \leq -2$ is equivalent to the statement
$|E(G)| \geq (k-1  + \frac{1}{\lambda})\frac{|V(G)|}{2} + \frac{1}{\lambda}$. Therefore, 
since $4 K_2$ and double cycles are not graphs,
Theorem~\ref{thm:stronger} implies Theorem~\ref{thethm}.

A multigraph is {\em exceptional} if it is either a complete graph, or $4 K_2$ or a double cycle in the case  $k=5$.

We end this section with some lemmas for the potential function.
The following well-known lemma (c.f.~\cite[Fact 11]{2018KoYa}) shows that our potential function is submodular.
\begin{lemma}\label{lem:submodularity}
    If $U_1, U_2 \subseteq V(G)$, then
    $\rho_{G,h}(U_1 \cup U_2) + \rho_{G,h}(U_1 \cap U_2) \leq \rho_{G,h}(U_1) + \rho_{G,h}(U_2)$.
\end{lemma}



Let $G^-$ denote a multigraph obtained from $G$ by removing any edge.

\begin{lemma}
\label{lem:Kr-sub-rho}
    Let $F$ be a copy of $K_k$ and let $h(v) \in \{k-2,k-1\}$ for each $v \in V(F)$, where at most one vertex $v$ has $h(v)=k-2$. 
    For a proper induced subgraph $F'$ of $F^-$,
    $\rho_h(F') \geq \rho_{h}(F^-) + (k-3)(\lambda -1)-2$.    
\end{lemma}
\begin{proof}
    The potential of a complete graph $K_j$ having  $j_1 \in \{j-1,j\}$ vertices $v$ with $h(v)=k-1$ and 
     $j_2 = j-j_1$
    vertices $u$ with $h(u)=k-2$
    is 
        \[
        j_1((k-1)\lambda +1)+j_2((k-2)\lambda -1) - 2\lambda  \binom j2 = j k \lambda - j \lambda - j_2 \lambda + j_1 - j_2 - \lambda j (j-1)
        \]
        \[
        = -\lambda j^2 + j k\lambda + j_1 - j_2(\lambda + 1).
        \]
    In order to minimize the potential of $F'$, we assume that $F'$ is a complete graph.
    If $|V(F')| = j$, then 
    $\rho_h(F') = -\lambda j^2 + j(\lambda k+1)$ when $j_2=0$ and $\rho_h(F') = -\lambda j^2 + j(\lambda k+1)-\lambda-2$ when $j_2=1$.
   In both cases, as a function of $j$, $\rho_h(F')$ is a concave quadratic function and hence is minimized at $j=1$ or $j=k-1$. 
    Therefore, if $j_2=0$, then
    \[\rho_h(F') \geq \min \{ (k-1)\lambda +1, (k-1)(\lambda +1)\} = (k-1)\lambda +1,\]
  and if $j_2=1$, then
    \[\rho_h(F') \geq \min \{ (k-2)\lambda -1, (k-2)(\lambda +1) - 1\} = (k-2)\lambda -1.\]
 Thus, in both cases,
\[
\rho_h(F') \geq (k-1)\lambda + 1 - j_2(\lambda + 2).
\]
As $\rho_h(F^-) = k - j_2(\lambda+2) + 2\lambda $,
we get 
\[\rho_h(F') - \rho_h(F^-) \geq  (k-1)\lambda + 1 - k  - 2\lambda = (k-3)(\lambda -1)-2.
\]

\end{proof}

\begin{lemma}\label{lem:rho5}
Let $k=5$, and let $F$ be an exceptional multigraph such that $h(v) = 4$ for each $v \in V(F)$. 
Then $\rho_h(F)\in\{5, 0, -1\}$. 
Moreover, if $A$ is a non-empty proper subset of $V(F^-)$, then $\rho_{F^-}(A)\geq \rho_h(F^-)+8$.
\end{lemma}

\begin{proof}
    As $k = 5$, $\lambda = 6$, so, $(k-1)\lambda + 1= 25$ and $2 \lambda = 12$. 
    If $F = K_5$, then $\rho_h(F) = 5(25) - 10(12) = 5$. If $F = 4 K_2$, then $\rho_h(F) = 2(25) - 12 - 3(13) = -1$. If $F = 2 C_n$, then $\rho_h(F) = 25n - 12n - 13n = 0$.

Let $A$ be a nonempty proper subset of $V(F^-)$.
We consider the following cases.
        If $F=4 K_2$, then $\rho_h(F^-)=12$, and $\rho_{h}(A)=25$. 
        If $F=K_5$, then $\rho_h(F^-)=17$, and $\rho_{F^-}(A) \geq 25|A| - 12 \binom{|A|}{2} \geq 25$.
    If $F$ is a double cycle, then $\rho_h(F^-)=13$ and $\rho_{F^-,h}(A)\geq 25 |A| - (12+13)(|A| - 1) = 25$.
\end{proof}

\section{Properties of GDP-forests}\label{sec:prelim}

For a positive integer $q$, the \emph{$q$-blowup} of a graph $G$ is the multigraph obtained by replacing each $v \in V(G)$ with an independent set $I_v$ of size $q$ and replacing each edge $uv \in E(G)$ with a $K_{q,q}$ joining $I_u$ and $I_v$. 

Suppose that $(H,L)$ is a DP-cover of a multigraph $G$.
The \emph{$q$-blowup} of $(H,L)$ is a DP-cover of $q G$ obtained by replacing each vertex $u \in V(H)$ with an independent set $I_u$ of size $q$, replacing each edge $uw \in E(H)$ with a $K_{q,q}$ joining $I_u$ and $I_w$, and replacing each list $L(v)$ with the set $\bigcup_{u \in L(v)} I_u$.

The following lemma is a partial case of a result by Kim and Ozeki~\cite[Theorem~5]{KimO}, see also~\cite[Theorem~1.17]{2024StScTo}.

\begin{lemma}
\label{lem:clique-cover}
Let $G$ be a multigraph and $h$ be a function on $V(G)$. 
Let $(H, L)$ be an $h$-cover of $G$.

(a) If $G=q K_{t}$ and $h(v)=(t-1)q$ for all vertices $v$, then $G$ has no $(H, L)$-coloring if and only if $(H,L)$ is a $q$-blowup of a $(t-1)$-cover $(H',L')$ of 
$K_t$ where $H'$ forms 
 $t-1$ disjoint copies 
of $K_t$.

(b) If $G=q C_{2t}$ and $h(v)=2q$ for all vertices $v$, then $G$ has no $(H, L)$-coloring if and only if $(H,L)$ is  a $q$-blowup of a $2$-cover $(H',L')$ of $C_{2t}$ where $H'$ is $C_{4t}$.

(c) If $G=q C_{2t+1}$ and $h(v)=2q$ for all vertices $v$, then $G$ has no $(H, L)$-coloring if and only if $(H,L)$ is 
a $q$-blowup of a $2$-cover $(H',L')$ of $C_{2t+1}$ where $H'$ forms
 two disjoint copies of $C_{2t+1}$.
\end{lemma}

For the next lemma, we need some more definitions and notation.
Let ${\rm dist}_G(u,v)$ denote the distance from $u$ to $v$ in a multigraph $G$.
The {\em diameter}, ${\rm diam}(G)$, of a connected multigraph $G$ is $\max_{u,v\in V(G)}{\rm dist}_G(u,v)$.
A vertex $x\in V(G)$ is {\em peripheral} if there is $y\in V(G)$ such that ${\rm dist}_G(u,v)={\rm diam}(G)$. The {\em block tree}
of a connected multigraph $G$ is the bipartite graph $H$ whose parts are the set of cut vertices and the set of the blocks of $G$, with a cut vertex $v$ adjacent to block $B$ in $H$ if and only if $v\in B$. This definition implies that the block tree of a connected multigraph $G$ is always a tree whose leaves are some blocks of $G$.
Call a block of a connected multigraph $G$ {\em peripheral} if the corresponding vertex in the block tree of $G$ is peripheral.

For a multigraph $F$, let $\wt{F}$ denote the {\em underlying graph of $F$}, i.e. the graph from which $F$ is obtained by multiplying some edges.
In other words, $\wt F$ is the graph on $V(F)$ in which $u$ and $v$ are adjacent if and only if $|E_F(u,v)| \geq 1$.
For a vertex $v\in V(F)$, let $\wt{d}(v)$ denote $d_{\wt F}(v)$, the degree of $v$ in $\wt{F}$, which is equal to $|N_F(v)|$.

For a multigraph $F$, let $m(F):=|E(F)|-|E(\wt F)|$.
 So, $2m(F)=\sum_{v\in V(F)}(d_F(v)-\wt d_F(v))$.

Given a function $h: V(F) \to \N\cup\{0\}$ with $h(v)\geq d_F(v)$ for each $v\in V(F)$, define $\sigma_h(F):=\sum_{v\in V(F)}(h(v)-d_F(v))$.
For a positive integer $j$, let $V_j(F)$ and $V^-_j(F)$ denote the  sets of vertices $v\in V(F)$ with $h(v)=j$ and $h(v)<j$, respectively, in $F$.

\begin{lemma}\label{GDP'} Let $k\geq 5$ and $\alpha=\frac{k-2}{2k-7}$.
Suppose that $T$ is a GDP-tree and $h: V(T)\to \N\cup\{0\}$ satisfies

(i) $3\leq h(v)\leq k-1$ for each $v\in V(T)$,

(ii) $h(v)\geq d_T(v)$ for each $v\in V(T)$, and

(iii) $T$ has neither $(k-1)$-regular nor $(k-2)$-regular blocks.

 Define $\Phi_k(T)= \alpha \sigma_h(T)+m(T)+|V^-_{k-1}(T)|-|V_{k-1}(T)|$.
 Then 
 \begin{equation}\label{s(T)'}
\Phi_k(T)>1+\alpha.
 \end{equation}
\end{lemma}

\begin{proof} We use induction on the number of blocks in $T$. 

Suppose first that $T$ has exactly one block. 
If $T$ is a single vertex $v$, then $d_T(v)=0$.
So, if $h(v)\leq k-2$, then $V_{k-1}(T)=\emptyset$ and
 by (i), $\sigma_h(T)= h(v)-d_T(v) \geq 3$,  implying that $\Phi_k(T) = \alpha(h(v) - d_T(v)) +1\geq 3 \alpha+1>1+\alpha$.
 Otherwise,
 $|V_{k-1}(T)|=1$ and
  $$\Phi_k(T)=\alpha(h(v)-d_T(v))-1= \alpha(k-1) -1 =\alpha-1+\frac{(k-2)^2}{2k-7}> 1+\alpha$$ for $k \geq 5$, again implying~\eqref{s(T)'}.
Suppose $T$ has two vertices, say $v_1$ and $v_2$, and $j\geq 1$ edges, so $d_T(v_1)=d_T(v_2)=j$, $|V(T)|=2$ and $m(T)=j-1$.
By (i)--(iii), $j\leq k-3$.
We can write $\Phi_k(T)$ in  the form $\Phi_k(T)=\sum_{i=1}^2 \Phi'_k(v_i)$, where
$$\Phi'_k(v_i)=\alpha(h(v_i)-d_T(v_i))+\frac{j-1}{2}+|V_{k-1}^-(T)\cap \{v_i\}|-|V_{k-1}(T)\cap \{v_i\}|.
$$
For each $i \in \{1,2\}$,
if $h(v_i)=k-1$, then since $j\leq k-3$ and $k\geq 5$,
\begin{eqnarray*}
\Phi'_k(v_i)&=& \alpha (k-1-j)+\frac{j-1}{2}+0-1=\alpha(k-1)-\frac{3}{2}-\frac{3j}{4k-14} \\
&\geq & \alpha(k-1) -\frac{3}{2}- \frac{3k-9}{4k-14}=\frac{2k^2-15k+34}{2(2k-7)}>  1.
\end{eqnarray*}
If $h(v_i)<k-1$, then 
\[\Phi'_k(v_i)\geq
\begin{cases}
    \displaystyle
    0+\frac{j-1}{2}+1-0\geq\frac{3}{2} & \mbox{if $j\geq 2$}\\[2ex]
\alpha (3-j)+0+1-0>2 & \mbox{if $j=1$}
\end{cases}\]
Thus in all cases, $\Phi_k(T)=\sum_{i=1}^2 \Phi'_k(v_i)> 2\geq 1+\alpha,$
as claimed.

 If $T$ has more than two vertices, then $T$ 
 is a multiple of a cycle or a complete graph.
 In particular, $T$ is $j$-regular for some $1 \leq j\leq k-3$. Similarly to above, we can write
 $\Phi_k(T)=m(T)+\sum_{v\in V(T)} \Phi_k(v)$, where
 $$\Phi_k(v)=\alpha (h(v)-d_T(v))+|V_{k-1}^-(T)\cap \{v\}|-|V_{k-1}(T)\cap \{v\}|.$$
For every $v\in V_{k-1}^-(T)$, 
$\Phi_k(v)\geq 0+1-0=1.$
If $j\leq k-4$, then for every $v\in V_{k-1}(T)$, 
\begin{equation}\label{k-1}
\Phi_k(v) = \alpha (k-1-j)+0-1\geq 3\alpha-1
\end{equation}
In this case, $\Phi_k(T) \geq |V(T)|\geq 3\cdot\min\{1,3\alpha-1\} >  1+\alpha$, and we are done. Thus, we suppose $j=k-3$.
Then instead of~\eqref{k-1}, we get
\begin{equation}\label{k-1'}
\Phi_k(v) = \alpha (k-1-j)+0-1 =2\alpha-1=\frac{3}{2k-7}>0.
\end{equation}

So if
 $|V_{k-1}^-(T)|+m(T)\geq 2$, then 
$\Phi_k(T) >  2>1+\alpha$, as claimed. Otherwise,
$G$ is a $(k-3)$-regular graph and $|V_{k-1}^-(T)|\leq 1.$
It follows that $G=K_{k-2}$ and by~\eqref{k-1'},
$$\Phi_k(T) \geq  (k-2)\left(2\alpha -1
\right)= (k-2) \frac{3}{2k-7}= 3\alpha >1+\alpha.  $$
This proves the base of induction.

\medskip

Now, let $b \geq 1$, and suppose that the lemma holds for GDP-trees with at most $b$ blocks, and that $T$ has $b+1\geq 2$ blocks.
Choose a peripheral block $B$ in $T$ so that the cut vertex $u$ in $B$ has the minimum number of neighbors in $B$. 
 Let $s=|V(B)|-1$.
 
Let $T'=T-(V(B)-u)$. Then $T'$ is a GDP-tree satisfying (i)--(iii) and has fewer blocks than $T$.
So, by induction, $\Phi_k(T')>1+\alpha$. 
Since $T$ is a GDP-tree, $B$ is $j$-regular for some $1\leq j\leq k-3$. 
In particular, $d_T(u)-d_{T'}(u)=j$, and the  multiplicity of every existing edge in $B$ is the same, say $q$.

Suppose first that $s=1$, so that $V(B)=\{u,v\}$. If $v\in V_{k-1}^-(T)$,
then $\Phi_k(T)\geq \Phi_k(T')-\alpha j+(j-1)+1-0\geq \Phi_k(T')$. 
If  $v\in V_{k-1}(T)$, then since $j\leq k-3$, 
$$\Phi_k(T)- \Phi_k(T')= -\alpha j+(j-1)-1+\alpha(k-1-j)=
j(1-2\alpha)-2+\alpha(k-1)$$
$$\geq (k-3)(1-2\alpha)-2+\alpha(k-1)=-\frac{3(k-3)}{2k-7}-2+\frac{(k-1)(k-2)}{2k-7}=
\frac{k^2-10k+25}{2k-7}\geq 0.$$

On the other hand, suppose $s\geq 2$. Then
\begin{equation}\label{2j}
\mbox{$\wt B$ is $t$-regular for some $t\geq 2$, and $j=qt\geq 2q$.}
\end{equation}
For each $v\in (V(B)-u)\cap  V_{k-1}(T)$,
$v$ contributes $\alpha (k-1-j)+0-1$ to
$\Phi_k(v)$, 
and for each $v\in (V(B)-u)\cap  V_{k-1}^-(T)$,  $v$ contributes at least $1$ to $\Phi_k(v)$.
For two adjacent vertices $x,y\in V(B)$, the $q$ edges in $E_B(x, y)$ contribute $q-1$ to $\Phi_k(T)- \Phi_k(T')$ due to the increase in $m(T)$.
Moreover, for each neighbor $x\in V(B)$ of $u$,  the $q$ edges in $E_G(u,x)$ contribute an additional $-\alpha q$ to $\Phi_k(T)- \Phi_k(T')$ due to the decrease in $h(u)-d_T(u)$.
Note that $s\geq \frac{j}{q}$,
and $B$ has at least $(s+1)q$ edges, so $m(B)\geq (s+1)(q-1)$.

If $j\leq k-4$, then by~\eqref{2j}, $k\geq 4+j\geq 4+t\geq 6$. In this case,
each vertex $v \in V(B) - u$ contributes at least 
$\min\{\alpha(k-1-j)-1,1\}\geq\min\{3\alpha-1,1\}\geq\alpha$ to $\Phi_k(T)$. Therefore,
\begin{eqnarray*}
\Phi_k(T)-\Phi_k(T')&\geq& s\alpha+(s+1)(q-1)-\frac{j}{q}(q\alpha) 
=  \alpha(s-j)+(s+1)(q-1) \\  &\geq& \alpha(s-j)+j-s+q-1=(j-s)(1-\alpha)+q  -1 \geq 0,
\end{eqnarray*}
    since the third 
    expression is nonnegative when $s \geq j$, and the last expression is nonnegative when $j \geq s$.

Otherwise, $j=k-3$, so
each vertex $v \in V(B) - u$ contributes at least 
$\min\{\alpha(k-1-j)-1,1\}=\min\{2\alpha-1,1\}=2\alpha-1$ to $\Phi_k(T)$.
If $q\geq 2$, 
then $j  = qt \geq 4$.
Then,  by~\eqref{2j}, $k=3+j \geq 7.$
If $s=2$, then $\wt{B}$ is a $3$-cycle, so $j$ is even and $m(B)=\frac{3j}{2}-3\geq j-1$ since $j\geq 4$.
If $s\geq 3$, then 
$m(B) = \frac{j-t}{2}(s+1) \geq  \frac j4 (3+1) =  j$.
Hence, $m(B)\geq j-1$.

So, $m(B)-\alpha (d_T(u)-d_{T'}(u))\geq j-1-\alpha j=(j-1)(1-\alpha)-\alpha$.
Therefore, since $s\geq 2$ and $j\geq 3$, 
$$\Phi_k(T)- \Phi_k(T')\geq (j-1)(1-\alpha)-\alpha+s(2\alpha-1)\geq
2(1-\alpha)-\alpha+4\alpha-2=\alpha
>0.$$

The remaining possibility is that $j=k-3$ and $q=1$, which means  that  
$B=K_{k-2}$ or $k = 5$ and $B$ is a cycle. 
In the latter case, each $v \in T'$ contributes at least $1\cdot (3-2) = 1$, since $\alpha = 1, h(v)\geq 3$, and $j = 2$. Thus $\Phi_k(T)- \Phi_k(T') > 0$.
In the former case, we have
$d_{T'}(u) = d_T(u)-j\leq 2$, so write $N_{T'}(u)=\{u_1,u_2\}$ (possibly, $u_2=u_1$).

We say a tree $T_0$ is {\em good} if $T_0$ is a GDP-tree satisfying (i)--(iii) with fewer blocks than $T$ and $V(T_0)\subseteq V(T)-(V(B)-u)$.
 By induction, $\Phi_k(T_0)>1+\alpha$. 
 In the remainder of the proof, we will demonstrate $\Phi_k(T)-\Phi_k(T_0)\geq 0$, which proves the lemma. 
 We will often use the fact that 
 \begin{enumerate} \label{eqn:contribute}
 \item [(*)] 
  Each $v\in (V(B)-u)\cap V_{k-1}^-(T)$  contributes $1$ 
  to $|V_{k-1}^-(T)\cap \{v\}|$, and  \\
  each $v\in (V(B)-u)\cap V_{k-1}(T)$ contributes $2\alpha-1$
   to $\alpha (h(v)-d_{T}(v) )-|V_{k-1}(T)\cap \{v\}|$,
 \end{enumerate}
so the vertices in $V(B)-u$ collectively contribute at least $(k-3)\min\{1,2\alpha-1\}=(k-3)(2\alpha-1)$ to $\Phi_k(T)$. 

 Write $T'' = T'- u$ and observe that $T''$ is a good tree.
Since $B$ is peripheral, $u_1$ and $u_2$ are in the same block, say $B_0$. 
Hence, $B_0$ is either a $K_2$ or a cycle.

\medskip
{\bf Case 1:} Assume $u_1=u_2$, so $N_{T'}(u)=\{u_1\}$. 
Note that $u$ contributes no less than $-1$ to $\Phi_k(T)$.

If $d_T(u)=k-2$, then $h(u_1)-d(u_1)$ decreases exactly by $1$. Together with (*),
$$\Phi_k(T)- \Phi_k(T'')\geq -\alpha-1 +(k-3)(2\alpha-1) 
= \alpha (2k-7)-(k-2)= 0.
$$

If $d_T(u)=k-1$, then $m(B_0)=1$ and $d_T(u_1)-d_{T''}(u_1)=2$.
 Together with (*),
$$\Phi_k(T)- \Phi_k(T'')\geq -2\alpha+1-1 +(k-3)(2\alpha-1)
= \alpha(2k-8)-(k-3)
= \frac{k-5}{2k-7} \geq  0.
$$

For the remaining cases, assume $u_1\neq u_2$, so $d_T(u)=k-1$. 
As $B$ is chosen so that the cut vertex of $B$ has minimum degree in $B$ over all cut vertices of peripheral blocks of $T$,
$B_0$ is not peripheral, and thus $T$ has at least three blocks.

\medskip
{\bf Case 2.1:} $|V(B_0)|\geq 4$, so $B_0$ is a cycle of length at least $4$, and $u_1u_2$ is not an edge.
Obtain $T_1$ from $T''$ by adding the edge $u_1u_2$, so $T_1$ is a good tree. 
Then $u$ contributes $-1$ to $\Phi_k(T)- \Phi_k(T_1)$.  
Together with (*),
$$\Phi_k(T)- \Phi_k(T_1)\geq -1 +(k-3)(2\alpha-1) 
= \alpha (2k-6)-(k-2)> 0.
$$

\medskip
{\bf Case 2.2:} $|V(B_0)|=3$ and $d_T(u_1)=2$. 
Obtain $T_2$ from $T''$ by deleting $u_1$, so that $T_2$ is a good tree. 
Then  $d_T(u_2)-d_{T_2}(u_2)=2$, and $u_1$ contributes to $\Phi_k(T)- \Phi_k(T_2)$ either $\alpha (k-1-2)-1$ (when $h(u_1)=k-1$) or at least $1$ (when $h(u_1) \leq k-2$). 
As in Case 2.1, $u$ contributes $-1$ to $\Phi_k(T)- \Phi_k(T_2)$. 
Together with (*),
$$\Phi_k(T)- \Phi_k(T_2)\geq -2\alpha +(3\alpha-1)-1 +(k-3)(2\alpha-1)
= \alpha (2k-5)-(k-1)=\frac{3}{2k-7}> 0.
$$

By the symmetry between $u_1$ and $u_2$, the last case is as follows.
\medskip

{\bf Case 2.3:}  $|V(B_0)|=3$ and both $u_1$ and $u_2$ are cut vertices. 
For $i\in\{1,2\}$, let $B_i$ be a block of $T$ containing $u_i$ and distinct from $B_0$. Since $B$ is peripheral, at least one of $B_1,B_2$ is also peripheral, say $B_1$. Then by the choice of $B$,
$B_1$ is $(k-3)$-regular, and so by the previous cases, $B_1=K_{k-2}$.
In this case, $T_3=T-V(B)-V(B_1)$ is a good tree. 
Then  $d_T(u_2)-d_{T_3}(u_2)=2$, and each of $u,u_1$ contributes $-1$ to $\Phi_k(T)- \Phi_k(T_2)$.  
Together with (*) for both $B$ and $B_1$,
$$\Phi_k(T)- \Phi_k(T_3)\geq -2\alpha-2  +2(k-3)(2\alpha-1) 
= 2\alpha (2k-7)-2(k-2)= 0.
$$
This proves the lemma.
\end{proof}

\section{Properties of a minimum counterexample $G$}\label{sec:propG}

Suppose that Theorem~\ref{thm:stronger} does not hold for some $k\geq 5$. Then there are a multigraph $G$, a function $h: V(G)\to [0,k-1]$, and an $h$-cover $(H,L)$ 
such that

\begin{enumerate}[(A)]
    \item $G$ has no  $(H,L)$-coloring, but every proper subgraph $G'$ of $G$ is DP $h\vert_{V(G')}$-colorable,
 and
    \item none of the claims (1)--(3) of Theorem~\ref{thm:stronger} holds for $G$ and $h$.
\end{enumerate}
 
Among such quadruples $(G,h,H,L)$, choose one    
   with minimum $|G|$ and fix it for the rest of the paper. 

For every $uv\in E(G)$, we assume $E_H(L(u), L(v))$ is the union of $|E_G(u, v)|$ maximal matchings.
We write $L(v) = \{1_v, 2_v, \dots, t_v\}$, where $t = |L(v)|$.

\subsection{General observations}

We make some observations about our minimum counterexample $G$ and the potentials of its subgraphs.

\begin{obs}
\label{obs:zero}
    $G$ has no vertex $v$ for which $h(v) = 0$. 
\end{obs}
\begin{proof}
    If $V(G) = \{v\}$, then $\rho_h(G) = \rho_h(v) = -2$, so $G$ is not a counterexample. If $|V(G)| \geq 2$, then $G[\{v\}]$ is a proper subgraph of $G$ that is not DP $h$-colorable, contradicting the assumption that $G$ is $h$-minimal.
\end{proof}

\begin{obs}
\label{obs:reducible}
    For each induced subgraph $X \subseteq G$, $X$ is not DP $(h-d_G + d_X)$-colorable.
\end{obs}
\begin{proof}
     As $G$ is $h$-minimal, $G-V(X)$ has an $(H,L)$-coloring $\varphi$. 
    For each $x \in V(X)$, let $L'(x) \subseteq L(x)$ consist of the colors of $L(x)$ that are not adjacent to a color of $\varphi$.
    Then, let $H'$ be the subgraph of $H$ induced by the set $\bigcup_{x \in V(X)} L'(x)$. As $G$ has no $(H,L)$-coloring,  $\varphi$ cannot be extended to $X$; therefore, $X$ has no $(H',L')$-coloring. As $|L'(x)| \geq h(x) - d_G(x) + d_X(x) $ for each $x \in V(X)$,  $X$ is not DP $(h-d_G+d_X)$-colorable.
\end{proof}

\begin{obs}
\label{obs:list-degree}
    For each vertex $v \in V(G)$, $h(v) \leq d(v)$.
    
\end{obs}
\begin{proof}
We let $X = G[v]$. 
By Observation~\ref{obs:reducible},
$X$ is not DP $(h-d_G+d_X)$-colorable. This implies that $h(v) - d_G(v) + d_X(v) = h(v) - d(v) \leq 0$, completing the proof.
\end{proof}

We say that a vertex $v \in V(G)$ is \emph{low} if $h(v) = d(v)$.

\begin{lemma}
\label{lem:ERT}
    If $U \subseteq V(G)$ is a set of low vertices, then each block of $G[U]$ is isomorphic to $s K_{t}$ or $s C_{t}$, for some $t$ and $s$.
\end{lemma}
\begin{proof}
    Suppose that $G[U]$ contains a block $B$ that is isomorphic to neither $s K_{t}$ nor $s C_{t}$.
    Let $(H,L)$ be an $h$-cover of $G$ for which $G$ has no $(H, L)$-coloring.
    As $G$ is $h$-minimal, $G -V(B)$ has an $(H,L)$-coloring $\varphi$. 
    Now, for each $v \in V(B)$, let $L'(v) \subseteq L(v)$ consist of the colors in $L(v)$ with no neighbor in $\varphi$. Each $v \in V(B)$ is low, so $|L'(v)| \geq d_{B}(v)$.
    Then, $B$ has an $(H,L')$-coloring by Theorem~\ref{thm:deg-choosable}, and thus $G$ is $(H,L)$-colorable, a contradiction.
\end{proof}

\begin{lemma}
\label{lem:j(k-2)}
For each $j \geq 1$, if $S \subsetneq V(G)$ satisfies 
$|E_G(S,\overline S)| \geq j$, then
$\rho_{G,h}(S) \geq j(\lambda -2) + 1$. 
\end{lemma}
\begin{proof}
    Suppose that the lemma is false, and let $j$ be the smallest value for which the lemma does not hold.
    Then, $G$ has a set $S \subsetneq V(G)$ for which $\rho_{G,h}(S) \leq j(\lambda -2)$ and
    $|E_G(S,\overline S)| \geq j$.
    We choose $S$ to be a counterexample to the lemma with largest size.

Let $G' = G - S$. For each vertex $v \in V(G')$, let $h'(v) =\max\{0, h(v) - d_G(v) + d_{G'}(v) \} = \max\{0, h(v) - |E_G(v,S)|\}$.
By Observation~\ref{obs:reducible}, $G'$ is not DP $h'$-colorable.
Therefore, there exists $U \subseteq V(G')$ for which $G'[U]$ has a spanning $h'$-minimal subgraph. As $G$ is $h$-minimal, it follows that $(G,h)$ and $(G',h')$ do not agree on $U$; therefore, $U$ contains a neighbor $u$ of $S$. As $h'(u) < k-1$, and as $G$ contains no $K_k$ subgraph,
$\rho_{G',h'}(U) \leq -2$.

Now, consider the set $S':=U \cup S$, and write $\ell = |E_G(S,U)|$. As $U$ contains a neighbor of $S$, $\ell \geq 1$. 
 As $\rho_h(v)-\rho_{h'}(v)\leq (\lambda+2)|E_G(v, S)|$ for each $v\in U$, 
$\rho_{G,h}(U) \leq -2 + \ell (\lambda + 2)$.
We have 
\begin{equation}
\label{eqn:j(k-2)}
\rho_{G,h}(S') \leq 
\rho_{G,h}(U) + \rho_{G,h}(S) - 2\lambda \ell \leq -2 + \ell(\lambda + 2) - \ell(2\lambda) + \rho_{G,h}(S) = -2 - \ell(\lambda - 2) + \rho_{G,h}(S).
\end{equation}

Now, suppose $j = 1$.
Then, $\rho_{G,h}(S) \leq \lambda -2$, so
\[\rho_{G,h}(S') \leq -2 - \ell(\lambda -2) + \lambda -2 = -2 + (1-\ell)(\lambda -2) \leq -2.\]

Thus, it follows from the maximality of $S$ that $S' = V(G)$. Therefore, $\rho_h(G) \leq -2$, and $G$ is not a counterexample to Theorem~\ref{thm:stronger}, a contradiction. This completes the case that $j = 1$.

    Next, suppose that $j \geq 2$. Then,
    as $\rho_{G,h}(S) \leq j(\lambda -2)$,
    (\ref{eqn:j(k-2)}) implies that 
    \begin{equation}
    \label{eqn:S'}
    \rho_{G,h}(S')     \leq -2  - \ell (\lambda -2) + j(\lambda -2) 
    =-2 + (j - \ell)(\lambda -2).
    \end{equation}
    If $S' = V(G)$, then~\eqref{eqn:S'} implies that $\rho_h(G) \leq -2$, and $G$ is not a counterexample to Theorem~\ref{thm:stronger}.
    Otherwise, $|E_G(S', \overline{S'})| \geq 1$, so the $j=1$ case implies that $\rho_{G,h}(S') \geq \lambda -1$. In both cases by~\eqref{eqn:S'}, $\ell < j$, so that there is at least one edge in $E_G(S, V(G') \setminus U)$.

    Now, since $|E_G(S,U)| = \ell$ and $|E_G(S, \overline S)| \geq j$, it follows that $|E_G(S', \overline{S'})| \geq |E_G(S, V(G') \setminus U)| \geq j - \ell$.
    As $1 \leq j - \ell \leq j-1$,
     the minimality of $j$
     tells us that $\rho_{G,h}(S') > (j-\ell)(\lambda -2)$, contradicting (\ref{eqn:S'}). 
\end{proof}

We can make several observations from Lemma~\ref{lem:j(k-2)}.

\begin{obs}
\label{obs:geq-1}
Each vertex subset $U \subseteq V(G)$ satisfies $\rho_{h}(U) \geq -1$. 
\end{obs}
\begin{proof}
If $1 \leq |U| \leq |V(G) | - 1$, then $\rho_h(U) \geq \lambda -1$ by Lemma~\ref{lem:j(k-2)}. If $|U| = |V(G)|$, then $\rho_h(U) = \rho_h(G) \geq -1$ by our assumption that $G$ is a counterexample to Theorem~\ref{thm:stronger}.
\end{proof}

\begin{obs}
\label{obs:geqk-1}
    For every nonempty proper subset $S\subsetneq V(G)$, $\rho_G(S) \geq \lambda -1$.
\end{obs}
\begin{proof}
    As $S \subsetneq V(G)$ and $G$ is connected, $|E_G(S, \overline S)| \geq 1$, so the statement follows by setting $j=1$ in Lemma~\ref{lem:j(k-2)}.
\end{proof}

\begin{lemma}
\label{lem:k-1k}
    If $G$ has a nonempty proper subset $S\subsetneq V(G)$ for which $|E_G(S, \overline S)| = 1$, then $\rho_{G,h}(S) \in \{\lambda -1,\lambda \}$.
\end{lemma}
\begin{proof}
     Suppose that $E_G(S,\overline S) = \{x_0y_0\}$, where $x_0 \in S$ and $y_0 \in \overline S$.
    If we define $h':V(S) \to \N \cup\{0\}$ so that $h'(x_0) = h(x_0) - 1$ and $h'(v) = h(v)$ for each $v \in S \setminus \{ x_0 \}$, Observation~\ref{obs:reducible} implies that
    $G[S]$ is not $h'$-colorable,
    and hence $S$ has a subset $X$ for which
    $G[X]$ has a spanning $h'$-minimal subgraph. As $G$ is $h$-minimal, $h$ and $h'$ do not agree on $X$, so $x_0 \in X$. As $h'(x_0) < k-1$,
    and as $G$ contains no $K_k$ subgraph,
    $\rho_{G,h'}(X) \leq -2$.
     As $\rho_h(x_0) - \rho_{h'}(x_0) \leq \lambda + 2$, we have
    $\rho_{G,h}(X) \leq \rho_{G,h'}(X) + (\lambda +2) \leq -2 + (\lambda +2) =  \lambda $. 
    By a symmetric argument, $\overline S$ has a vertex subset $Y$ containing $y_0$ for which $\rho_{G,h}(Y) \leq \lambda $.

    Now, $\rho_{G,h}(X \cup Y) \leq \lambda + \lambda  -2\lambda  =0$. Therefore, Observation~\ref{obs:geqk-1} implies that $V(G) = X \cup Y$, further implying $X = S$ and $\rho_{G,h}(S) \leq \lambda $. As Observation~\ref{obs:geqk-1} implies that $\rho_{G,h}(S) \geq \lambda -1$, the proof is complete.
\end{proof}

\begin{lemma}\label{lem:no1}
    $G$ has no vertex $v$ satisfying $h(v) = 1$.
\end{lemma}
\begin{proof}
        If $h(v) = 1$, then $\rho_h(v) = \lambda -2$, contradicting Observation~\ref{obs:geqk-1}.
\end{proof}

\begin{lemma}
\label{lem:low}
    If $h(v) \in [2,\ldots, \frac{\lambda -1}{2}]$, then $v$ is low.
\end{lemma}
\begin{proof}
    We have $\rho_h(v) = h(v) \lambda  - 1 \leq \lceil \frac{h(v) \lambda  - 1}{\lambda -2} \rceil (\lambda -2) $. Then, Lemma~\ref{lem:j(k-2)} implies that $$d(v) \leq \left\lceil  
    \frac{h(v) \lambda  - 1}{\lambda -2} \right\rceil - 1 = h(v) + \left\lceil \frac{2h(v) - 1 }{\lambda -2} \right\rceil - 1 =h(v).$$
\end{proof}

\begin{lemma}
\label{lem:no2}
    $G$ has no vertex $v$ satisfying $h(v) = 2$.
\end{lemma}
\begin{proof}
    Suppose that $v \in V(G)$ satisfies $h(v) = 2$.
    By Lemma~\ref{lem:low}, $d(v) = 2$.
    Suppose first that $v$ is joined to a neighbor $w$ by a pair of parallel edges. 
    Then, $N(v) = \{w\}$.

    Now, 
    let $G' = G-v$. Choose a color $c \in L(v)$, and obtain $H'$ from $H$ by deleting $L(v) \cup N_H(c)$. We may assume without loss of generality that $|L(w) \cap N_H(c)| \geq 2$.
    Let $L'(w) = L(w) \cap V(H')$, and let $L'(x) = L(x)$ for each $x \in V(G') - w$. If $G'$ has an $(H',L')$-coloring $\varphi'$, then we can extend it to $G$ by letting $\varphi'(v)=c$.
    So assume $G'$ has no $(H',L')$-coloring. 
    Then there exists a subset $U' \subseteq V(G')$ for which $G'[U']$ has a spanning $h'$-minimal subgraph. As $G$ is $h$-minimal and $U' \subsetneq V(G)$, it follows that $h'$ and $h$ do not agree on some vertex of $U'$; therefore, $w \in U'$.  As $h'(w) \leq k-3$, it follows from the minimality of $G$ that $\rho_{G',h'}(U') \leq -2$, so that $\rho_{G,h}(U') \leq -2 + \lambda  + (\lambda +2) = 2\lambda $. Then,
    \[\rho_{G,h}(U' + v) \leq 2\lambda  - (4\lambda +1) + (2\lambda -1) = -2,\]
    contradicting Observation~\ref{obs:geq-1}.

Otherwise, suppose $N(v)=\{w,u\}$, where $u\neq w$.
    By Lemma~\ref{lem:no1}, $|L(u)| \geq 2$ and $|L(w)| \geq 2$. 
We may assume that $\{1_v 1_u, 2_v 2_u, 1_v 1_w, 2_v 2_w \} \subseteq E(H)$.

Obtain $G'$ from $G-v$ by adding an edge $uw$.
Also, let $(H',L)$ be an $h$-cover of $G'$ obtained by deleting $L(v)$ and then adding 
 the matching $\{1_u2_w,2_u1_w\}$. 
 (Here, we consider $h$ restricted to $V(G')$.)
If $G'$ has an $(H',L)$-coloring $\varphi$, then we can extend $\varphi$ to $v$, because if $\varphi(u)$ forbids the color $i_v$ for $v$, then (due to our new edge $uw$) $\varphi(w)$ does not forbid the color $(3-i)_v$ for $v$. 
 Thus, $G'$ has an $h$-minimal subgraph $G''$
 containing $u$ and $w$. Let $U=V(G'')$.

 As $|V(G')| < |V(G)|$, it follows from our induction hypothesis that $\rho_{G',h}(U) \leq -2$ or  $G''$ is exceptional.
 We consider several cases.

 First, suppose that 
 $k=5$, $h(x) = 4$ for each $x \in U$, and that $G''$ is a copy of $4 K_2$ or $2 C_t$ for some integer $t$.
 Since each adjacent pair of vertices in $G''$ has edge multiplicity at least $2$,
 and as $u,w \in U$,
 we have $uw \in E(G)$.
 Furthermore, as $G''$ is exceptional, 
 $\rho_{G'',h}(U) \leq 0$.
 Therefore,
 \[\rho_{G,h}(U \cup \{v\}) \leq 0 + (2 \lambda + 1) - 2(2\lambda) + (2\lambda - 1) = 0.\]
 Thus, Observation \ref{obs:geqk-1} implies that $V(G) = U \cup \{v\}$.
Furthermore, Observation \ref{obs:geq-1} 
implies that $h(x) = 4$ for each $x \in U$.
Then, $G$ is DP $h$-colorable by Theorem \ref{thm:deg-choosable}, a contradiction.

Next, suppose that $\rho_{G',h}(U) \leq -2$. 

 Then $G[U\cup\{v\}]$ is constructed by deleting an edge $uw$ from $G'[U]$, adding a vertex $v$ with $h(v) = 2$, and adding two edges $vu$ and $vw$; hence,
 \[\rho_{G,h}(U \cup  \{v\}) \leq 
 -2 + (2\lambda -1) - 2(2\lambda ) + (2\lambda +1) = -2,\]
 contradicting Observation~\ref{obs:geq-1}. 
 
 Therefore, $G''$ is  a copy of $K_k$, and for some $\varepsilon\in \{0,1\}$, $h(x) = k-1$ for
  $k-1-\varepsilon$ vertices $x\in U$, and $h(y) = k-2$ for $\varepsilon$ vertices $y \in U$.
 Then, writing $U' = U \cup \{v\}$,
 \[\rho_{G,h}(U') = 
 k + (2\lambda - 1) - 2(2\lambda) + 2 \lambda-\varepsilon(\lambda+2) = k-1-\varepsilon(\lambda+2).
 \]

 For $k \geq 5$, 
 \[2 \lambda - 3 = 2\left\lceil\frac{k^2-7}{2k-7}\right\rceil - 3 \geq  \frac{2k^2 - 14}{2k-7}  - 3 = k - 1 + \frac{3k}{2k-7} > k-1, \]
 so by Lemma~\ref{lem:j(k-2)}, 
 $j:=|E_G(U', \overline{U'})| \leq 1$ and if $\varepsilon=1$, then $j=0$.
 
If $j = 0$, then $U'=V(G)$, every vertex of $G$ is low, and $G$ is not a GDP-tree. Therefore, $G$ is DP $h$-colorable, a contradiction. 
If $j = 1$, then $\varepsilon=0$.
Write $x$ for the unique vertex of $U'$ with a neighbor outside of 
$U'$.
As $d(v) = 2$, $x \neq v$. 
If $x \not \in \{u,w\}$, then each vertex $z \in U' \setminus \{x\}$ is low, and $G'[U' \setminus \{x\}]$ is $2$-connected and not a GDP-tree. 
Hence, $G'[U' \setminus \{x\}]$ is DP $(h-d_G+d_{G[U' \setminus \{x\}]})$-colorable, contradicting Observation~\ref{obs:reducible}.

Otherwise, $x \in \{u,w\}$. 
We claim that $G[U']$ is DP $(h-d_G+d_{G[U']})$-colorable. For each $z \in U'$, we write $h'(z) = h(z) - d_G(z) + d_{G[U']}(z)$, and observe that $h'(x) = k-2 \geq 3$, and $h'(z)= h(z)$ for each $z \in U' \setminus \{x\}$.
Consider an $h'$-cover $(H',L')$ of $G[U']$.
We assign a color $c \in L'(x)$ to $x$ with no neighbor in $L(v)$; this is possible, as $|L(v)| = h'(v) = 2 < |L'(x)|$. Afterward, each vertex $z \in U' \setminus \{x\}$ has at least $d_{G[U' \setminus \{x\}]}(z)$ colors in $L'(z)$ not adjacent to $c$. 
Furthermore, $v$ has two colors in $L(v)$ not adjacent to $c$, while $v$ has only one neighbor in $U \setminus \{x\}$. 
{ We give a linear ordering to the vertices of $U' \setminus \{x\}$
that ends with $v$, and  observe that for each $z \in U' \setminus \{x\}$, the number of neighbors of $z$ appearing before $z$ in the ordering is less than $|L'(z) \setminus \{c\}|$.
Therefore,  by greedily coloring the vertices of $U' \setminus \{x\}$  according to our linear ordering, we complete an $(H',L')$-coloring of $G[U']$, contradicting Observation \ref{obs:reducible}.
}
Hence  we can finish an $(H',L')$-coloring of $G[U']$ using degeneracy,
contradicting Observation~\ref{obs:reducible}.
\end{proof}

\begin{cor}\label{cor:h-geq3}
    Every vertex $v$ in $G$ satisfies $h(v)\geq 3$.
\end{cor}
\begin{proof}
    This follows from Observation~\ref{obs:zero} and Lemmas~\ref{lem:no1} and~\ref{lem:no2}.
\end{proof}

\begin{lemma}
\label{lem:k-1-reg}
     Let $k\geq 5$.
    If $G$ is $2$-connected and $h(v) = k-1$ for each $v \in V(G)$, then $G$ is not $(k-1)$-regular.
\end{lemma}
\begin{proof}
 Suppose $G$ is $(k-1)$-regular.
    As $G$ is $h$-minimal, Theorem~\ref{thm:deg-choosable} implies that $G$ is either ${\frac{k-1}{t-1}} K_t$ for some $2 \leq t \leq  k-2$ (satisfying $(t-1) | (k-1)$) or ${(k-1)/2} C_n$ for some $n \geq 3$ (provided $k $ is odd). If $G$ is ${\frac{k-1}{t-1}} K_t$, then 
    \[\rho_h(G) = t((k-1)\lambda +1) - \binom t2 \left ( \frac{k-1}{t-1} \right ) (2\lambda +1) + \binom t2 = \binom{t+1}{2} - \frac 12 t(k-1) = \frac 12 t (t-k+2)  .\]
    As a function of $t$, this expression is a convex quadratic function and
    is maximized at $t=2$.
    When $t = 2$ and $k\geq 6$, we have $\rho_h(G) = 2 -  k+2 \leq -2$. 
    Similarly, if $k\geq 6$ and $G={(k-1)/2} C_n$ (so $k$ is odd), then 
\begin{eqnarray*}
\rho_h(G) &=& n((k-1)\lambda +1)  - n \left ( \frac {k-1}{2} \right )(2\lambda +1) + n 
= n \left (   \frac 52 - \frac 12 k \right )
\leq  -3
\end{eqnarray*}
for $k \geq 7$ and $n \geq 3$. 

In both cases, $\rho_h(G) \leq -2$, so $G$ is not a counterexample to Theorem~\ref{thm:stronger}, a contradiction. 

So suppose $k=5$.
If $G=2 C_n$, then Part 2 of Theorem~\ref{thm:stronger} holds.
If $G={\frac{5-1}{t-1}} K_t$ for some $2 \leq t \leq  5-2$, then $t\in \{2,3\}$. 
If $t=2$, then $G$ is isomorphic to $4 K_2$. If $t=3$, then $G$ is isomorphic to $2 C_3$.
In all cases, Part 2 of Theorem~\ref{thm:stronger} holds. 
\end{proof}

\subsection{A special set $S_0^*$ of vertices in $G$}

An {\em edge-block} in a multigraph $G$ is an inclusion maximal connected subgraph $G'$ of $G$ such that
either $|V(G')|=2$ or $G'$ has no cut edges. In particular, every cut edge forms an edge-block, and each connected multigraph decomposes into edge-blocks.

Define a special subset $S^*_0 \subseteq V(G)$ as follows. If $G$ has no cut edges, then $S^*_0=V(G)$. Otherwise, we fix a smallest pendent edge-block $B^*$ distinct from $K_2$
  and let $S^*_0=V(B^*)$. By Lemma~\ref{lem:no1} and Observation~\ref{obs:list-degree}, a cut edge cannot be a pendent edge-block, so $S^*_0$ is  well-defined.

    If $B^* \subsetneq G$, then
    since $B^*$ is pendent, there are $x^*_0\in S^*_0$ and $y^*_0\in V(G)-S^*_0$ such that $x^*_0y^*_0$ is the unique edge connecting $B^*$ with the rest of $G$. Fix these $B^*$, $S^*_0$, $x^*_0,y^*_0$.
  By definition, $B^*$ is $2$-edge-connected.

We have the following lemma.

\begin{lemma}
\label{lem:terminal}
Suppose that $S^*_0 \subsetneq V(G)$.
For each  $S \subseteq V(G)$ satisfying $\emptyset \subsetneq S \cap S_0^* \subsetneq S_0^*$, $\rho_{G,h}(S) \geq 2\lambda -3$.
\end{lemma}
\begin{proof}
If $G[S]$ is not connected, 
then each vertex set $S'$ inducing a component of $G[S]$ is a proper subset of $V(G)$, and hence $\rho_{G,h}(S') \geq \lambda-1$ 
by Lemma \ref{lem:j(k-2)}.
As $G[S]$ has at least two components, it follows that $\rho_{G,h}(S) \geq 2 \lambda - 2$. Thus, we assume that $G[S]$ is connected.

Recall that $x_0^*y_0^*$ is the cut edge joining 
$S_0^*$ and $\overline{S_0^*}$.
    First, suppose that $S \subsetneq S_0^*$.
    If $\rho_{G,h}(S) \leq 2\lambda -4$, then by Lemma~\ref{lem:j(k-2)}, $|E_G(S, \overline S)| \leq 1$. 
    Then, $S$ and $\overline S \cap S_0^*$ are joined by at most one edge, contradicting the $2$-edge-connectedness of $B^*$.

Next, suppose that $S$ is not contained in $S_0^*$, so that $x_0^*, y_0^* \in S$. 
We write $G' = G-S_0^*$ and observe that $\rho_{G,h}(V(G')) \leq \lambda $ by Lemma~\ref{lem:k-1k}.
We make several claims.

First, we claim that  $\rho_{G,h}(S \cap S_0^*)  + \rho_{G,h}(V(G')) \geq 4\lambda -3 $.
Indeed, as $S \cap S_0^*$ is joined to  $S_0^*-S$
by at least two edges, Lemma~\ref{lem:j(k-2)} implies that 
    \[2\lambda -3 \leq \rho_{G,h}(S \cup ( V(G')) \leq -2\lambda  + \rho_{G,h}(S \cap S_0^*) + \rho_{G,h}(V(G')).\]
    Rearranging, $\rho_{G,h}(S \cap S_0^*) + \rho_{G,h}(V(G')) \geq 4\lambda -3$.

    Next, we claim that $\rho_{G,h}(S \cap V(G')) \geq \rho_{G,h}(V(G'))$.
    Indeed, if $V(G') \subseteq S$, then the statement is clearly true. If $S \cap V(G') \subsetneq V(G')$, then as $G'$ is connected, $E_G(S \cap V(G'), \overline {S\cap V(G')})$ contains $x_0^* y_0^*$ as well as at least one additional edge of $G'$. Therefore, by Lemma~\ref{lem:j(k-2)}, $\rho_{G,h}(S \cap V(G')) \geq 2\lambda -3 > \rho_{G,h}(V(G'))$.

    Now, putting these claims together,
    \[\rho_{G,h}(S) = -2\lambda  + \rho_{G,h}( S \cap S_0^*) + \rho_{G,h} (S \cap V(G') )
    \geq -2\lambda  + \rho_{G,h}(S \cap S_0^*) + \rho_{G,h}(V(G')) \geq 2\lambda -3.\]
\end{proof}

\subsection{Subgraphs induced by low vertices in $G$}
Let $\BB$ be the subgraph of $G$ induced by low vertices in $G$ and let ${\BB}_0$ be the subgraph of $\BB$ induced by the vertices of $\BB$ in $S^*_0$.
By Lemma~\ref{lem:ERT}, each block in $\BB$ is a multiple of either a complete graph or a cycle.
In particular, each such block is regular.
Given a $(k-2)$-regular block $B \subseteq \BB$  along with distinct vertices $u,u' \in N(B) \setminus V(B)$, we write $F(B,u,u')$ for the multigraph obtained from $G - V(B)$ by adding an edge $uu'$.

\begin{lemma}\label{comp}
If $B$ is a $(k-2)$-regular block  of $\BB$, then the multigraph $F(B,u,u')$ is not DP $h$-colorable.
\end{lemma}
\begin{proof}
As $G$ is $h$-minimal, $G - V(B)$ has an $(H,L)$-coloring $\varphi$ that does not extend to all of $G$.
 For each $v \in V(B)$,
we write $L_\varphi(v)$ for the subset of $L(v)$ consisting of colors with no neighbor in $\varphi$.
We write $H' = H[L_\varphi(V(B))]$.
As $|L_\varphi(v)| \geq k-2$ for each $v \in V(B)$ and $B$ has no $(H',L_\varphi)$-coloring, Lemma~\ref{lem:clique-cover}
implies that $(H',L_\varphi)$ is 
one of the blowup covers described in Lemma~\ref{lem:clique-cover}. In particular, $H'$ is $(k-2)$-regular.  Thus,
$H_\varphi:=H[L(V(B))] - V(H')$ is isomorphic to a subgraph of $\wt B$.

Now, we consider two cases.
First, suppose that every component of $H[L(V(B))]$ is isomorphic to $H_\varphi$. Then, as at least one vertex $v \in V(B)$ satisfies $h(v) = k-1$ (namely the neighbor of $u$ in $B$), it follows that $H[L(V(B))]$ is isomorphic to $k-1$ disjoint copies of $H_\varphi$, each of which contains exactly one vertex from $L(v)$ for each $v \in V(B)$.
Now, we construct an $h$-cover $(H'',L)$ of $F(B,u,u')$
as follows. We begin with $(H,L)$, and we remove $L(V(B))$ from $H$. Then, we add a perfect matching between $L(u)$ and $L(u')$ so that $c \in L(u)$ and $c' \in L(u')$ are adjacent if and only if $c$ and $c'$ are adjacent to a common component of $H[L(V(B))]$.
Then, if $G - V(B)$ has an $(H'',L)$-coloring $\varphi'$, the colors $\varphi'(u)$ and $\varphi'(u')$ are adjacent to distinct components of $H[L(V(B))]$. 
Hence, there is no $H_\varphi$-component in $H[L(V(B))]$ all of whose vertices are adjacent to $\varphi'$; therefore, by Lemma~\ref{lem:clique-cover}, $\varphi'$ can be extended to all of $G$. This contradicts the initial assumption that $G$ is $h$-minimal, completing the first case.

By Lemma~\ref{lem:clique-cover}, the only other case to consider is that
$H[L(V(B))]$ has a single component isomorphic to $H_\varphi$.
In this case, $H'$ is a $q$-blowup of either a clique or a cycle, where $q \geq 2$ is the quotient of $k-2$ and the regularity of $\wt B$.
In particular, $\wt B$ has maximum degree less than $k-2$.
Then,
we construct an $h$-cover $(H'',L)$ of $F(B,u,u')$ as follows. 
We begin with $(H,L)$, and we remove $L(V(B))$ from $H$.
Then, we add a single edge to $H''$ joining $\varphi(u)$ and $\varphi(u')$. 

Now, suppose that $F(B,u,u')$ has an $(H'',L)$-coloring $\varphi'$. By construction, either $\varphi'(u) \neq \varphi(u)$ or $\varphi'(u') \neq \varphi(u')$. 
Therefore, some vertex of $H_\varphi$ has no neighbor in $\varphi'$. 
For each $v \in V(B)$, we let $L''(v)$ be the set of colors $c \in L(v)$ for which $c$ has no neighbor in $\varphi'$. 
Then, $H[L''(V(B))]$ is not $(k-2)$-regular, so by Lemma~\ref{lem:clique-cover}, the coloring $\varphi'$ extends to an $(H,L)$-coloring of $G$.
This contradicts our initial assumption that $G$ is $h$-minimal.

Therefore, $F(B,u,u')$ is not DP $h$-colorable.
\end{proof}

By Lemma~\ref{comp}, for each $(k-2)$-regular subgraph $B \subseteq \BB$,
and any two distinct vertices  $u,u'\in N(B) \setminus V(B)$, there is
an $h$-minimal multigraph 
$G(B,u,u')$  contained in $F(B,u,u')$.
As $G$ is $h$-minimal, $G(B,u,u')$ is not a subgraph of $G$; therefore, $G(B,u,u')$ contains $u$ and $u'$.

\begin{lemma}\label{notx0}
If $S^*_0 \neq V(G)$, then the vertex $x^*_0$ is not in a ${(k-2)}$-regular block of $\BB$.
\end{lemma}

\begin{proof} Suppose $x^*_0$ is in a ${(k-2)}$-regular block $B$ of $\BB$.
Since $B$ is $2$-connected, $y^*_0\notin V(B)$. 
If there exists a vertex  $y_1^*\in N(V(B)) \setminus (V(B) \cup \{y^*_0\})$, 
let $G'=G(B,y^*_0,y_1^*)$. Recall that  $y_0^*,y_1^*\in V(G')$. Since $y^*_0y_1^*$ is a cut edge in $G'$, it is not
part of an exceptional graph.
Therefore,
$\rho_{G',h}(V(G'))\leq -2$, and hence $\rho_{G,h}(V(G'))\leq 2\lambda -2$. 
However, the multigraph $G[V(G')]$ is disconnected. 
As $ G[S^*_0]$ is an edge-block of $G$, $|E_G(V(G') \cap S^*_0, S^*_0 \setminus V(G'))| \geq 2$; therefore, by Lemma~\ref{lem:j(k-2)}, $\rho_{G,h}(V(G') \cap S^*_0) \geq 2\lambda -3$. Hence, $\rho_{G,h}(V(G') \setminus S^*_0) \leq (2\lambda -2) - (2\lambda - 3) < \lambda - 1$, contradicting Observation~\ref{obs:geqk-1}.

Suppose now that $V(B)$ has no  neighbor in $V(G) \setminus (V(B) \cup \{y^*_0\})$. 
Then every vertex $u\in V(B)-x_0^*$ satisfies $h(u) = k-2$. 
Let $b = |V(B)|$, so
$$
\rho_{G,h}(V(B)) = -m(B)+(b-1)((k-2)\lambda  - 1) +((k-1)\lambda  + 1) - 2\lambda  b  \left (\frac{k-2}{2}  \right ) =-m(B)+\lambda +2-b.
$$
If $b\geq 4,$ then $\rho_{G,h}(V(B))\leq \lambda -2$,
contradicting Observation~\ref{obs:geqk-1}. Furthermore, if $2\leq b\leq 3,$
then $m(B)\geq 2$, so again $\rho_{G,h}(V(B))\leq \lambda -2$, contradicting Observation~\ref{obs:geqk-1}.
\end{proof}

\begin{lemma}\label{lem:NoKk-1}
If $k\geq 5$, then ${\BB}_0$ does not contain $K_{k-1}$ formed by  $(k-1)$-vertices. 
\end{lemma}
\begin{proof}
Suppose there is a set $C=\{v_1,v_2,\ldots,v_{k-1}\} \subseteq S^*_0$ consisting of low $(k-1)$-vertices such that $B:=G[C]=K_{k-1}$. 
For $i\in [1,k-1]$, let $v'_i$ be the unique neighbor of $v_i$ outside of $C$. 
Let $C'=\{v_1',v_2',\ldots,v_{k-1}'\}$, so that $v_i v_i' \in E(G)$ for each $i \in [1,k-1]$.
Since $G$ has no $K_k$, $|C'|\geq 2$. 

For distinct $v'_s,v'_t\in C'$,
write $G_{s,t} = G(B,v_s',v'_{t})$.
We will simultaneously consider two cases: If $\rho_h(G_{s,t}) \leq -2$ for some distinct $v'_s,v'_t\in C'$,
or if $G_{s,t}$ is exceptional for every pair of  distinct $v'_s,v'_t\in C'$.
Let $\varrho(C)=\min\{\rho_h(G_{s,t}): v'_s,v'_t\in C' \mbox{ and } v'_s\neq v'_t\}$.

For $j \in [1,k-1]$,  define the 
function
 \[ g(j) =  
 \begin{cases}
      2\lambda + 2j - 5 & \mbox{ if $\rho_h(G_{s,t}) \leq -2$ for some distinct  $v'_s,v'_t\in C'$  (Case 1)} \\
      \varrho(C) + 2 \lambda +j - 1  & \mbox{ if $G_{s,t}$ is exceptional for all distinct  $v'_s,v'_t\in C'$ (Case 2)}\\
 \end{cases}\]


For $j \in [1,k-1]$, let $\pi(j)$ be the minimum of $\rho_{G,h}(P)$ over $P \subseteq V(G) \setminus C$ with $|N(P) \cap C| \geq j$.
Let $P_j \subseteq V(G) \setminus C$
be such that $|N(P_j) \cap C| \geq j$ and $\rho_{G,h}(P_j) = \pi(j)$.

We will prove that $\pi(k-1) \leq g(k-1)$.
For this, let $j$ be the largest value for which $\rho_{G,h}(P_{j}) \leq g(j)$. 
If $j=k-1$, then we are done, so assume $ j < k-1$.
Then $j\geq 2$ since
 \[ \pi(2) =\rho_{G,h}(P_2)\leq  
 \begin{cases}
      -2+(2\lambda +1)=2\lambda-1=g(2). & \mbox{(Case 1)} \\
      \varrho(C) + (2 \lambda+1)=g(2).
      & \mbox{(Case 2)}\\
 \end{cases}\]

Assume there are distinct $i,i'\in [1,k-1]$ such that $v'_i\in P_{j}$ and $v'_{i'}\not \in P_{j}$. 
Since $V(G_{i,i'})\cup P_{j}\supseteq P_{j} \cup \{v'_{i'}\}$, 
$V(G_{i,i'})\cup P_{j}$ has at least $j+1$ neighbors in $C$.
Thus, Lemma~\ref{lem:submodularity} implies that
\begin{equation}\label{k4}
\pi(j+1) \leq \rho_{G,h}(P_{j}\cup V(G_{i,i'}))\leq g(j) +\rho_{G,h}(V(G_{i,i'}))-\rho_{G,h}(P_{j}\cap V(G_{i,i'})).
\end{equation}
In Case 2, by Lemma~\ref{lem:rho5}, $\pi(j+1)\leq g(j) - (k-3)(\lambda-1)+2 < g(j+1)$, contradicting the maximality of $j$. 
In Case 1, $\rho_h(G_{i,i'})\leq -2$, so as both $P_{j}$ and $V(G_{i,i'})$ contain $v'_i \in S_0^*$, by Lemma~\ref{lem:terminal}, $\rho_{G,h}(P_{j}\cap V(G_{i,i'}))\geq 2\lambda -3$. 
Thus, $\pi(j+1) \leq g(j) + (-2 + 2\lambda + 1) - (2\lambda - 3) = g(j) + 2 = g(j+1)$, again contradicting the maximality of $j$.
Therefore, $C'\subseteq J_{k-1}$, so $\pi(k-1)\leq g(k-1)$.

Since $\rho_{G,h}(C) = (k-1) \left ( (k-1)\lambda +1 \right )  - 2\lambda  \binom{k-1}{2} = (\lambda +1)(k-1)$, we get
 \[ \rho_{G,h}(C\cup P_{C'}) 
 {\leq}\rho_{G,h}(C)+g(k-1)-2\lambda(k-1)
 {=}(\lambda+1)(k-1)+g(k-1)-2\lambda(k-1)
 {=}(1-\lambda)(k-1)+g(k-1)\]
$$=\begin{cases}
      (1-\lambda)(k-1)+(2\lambda+2(k-1)-5)\leq-2. & \mbox{  (Case 1)} \\
      (1-\lambda)(k-1)+(\varrho(C)+2\lambda+k-1) \\
      \hspace{1.2in}\leq(1-\lambda)(k-1)+2\lambda+2k-2\leq-3.& \mbox{  (Case 2: $k\geq 6$)}\\
      (1-\lambda)(k-1)+(\varrho(C)+2\lambda+k-1)\\
      \hspace{1.2in}\leq (1-\lambda)(k-1)+2\lambda+k-1\leq-4. & \mbox{  (Case 2: $k=5$ and $\varrho(C)\leq 0$)}\\
      (1-\lambda)(k-1)+(\varrho(C)+2\lambda+k-1)\\
      \hspace{1.2in}= (1-\lambda)(k-1)+2\lambda+k+4=1. & \mbox{  (Case 2: $k=5$ and $\varrho(C)=5$)} 
 \end{cases}$$

 This contradicts Observation~\ref{obs:geq-1} in all cases except when $k=5$ and $\varrho(C)=5$.
 Since $1<\lambda-1$ when $k=5$, Observation~\ref{obs:geqk-1} implies $C\cup P_{C'}=V(G)$.
 In particular, $G[P_{C'}]$ is $K_5^-$ 
 and $G[C']$ is $K_4^-$ so $G_{s,t}$ is exceptional only for one pair $s,t$.
 Yet, this is impossible since we are in Case 2.
\end{proof}

\begin{lemma}\label{Nothing}
${\BB}_0$ does not contain a ${(k-2)}$-regular block.
\end{lemma}

\begin{proof} Suppose that ${\BB}_0$ contains $C=\{v_1,v_2,\ldots,v_{b}\}$
such that $B:=G[C]$ is a ${(k-2)}$-regular block.   
Let $M$ be the set of $(k-2)$-vertices in $C$.
For all $v_i\in C-M$, let $v'_i$ be the unique neighbor of $v_i$ outside of $C$.
Since $G$ is $h$-minimal,
$G-C$ has an $(H, L)$-coloring $\varphi$. 
We write $H_\varphi$ for the set of colors in $H$ above $B$ that are adjacent to $\varphi$. 
For each $v_i \in C$, we write $L_\varphi(v_i) = L(v_i) \setminus V(H_\varphi)$.
Since $|L_\varphi(v)| = d_B(v)$ for each $v \in V(B)$
and $B$ is not $(H[L_\varphi(C)],L_\varphi)$-colorable, Theorem~\ref{thm:deg-choosable} implies that $B$ is a multiple of either a cycle of a complete graph.

If $B$ is isomorphic to $K_{k-1}$, then by Lemma~\ref{lem:clique-cover}, $H[L_\varphi(C)]$ consists of $k-2$ copies of the graph $K_{k-1}$.
In particular, $H[L_\varphi(C)]$  is $(k-2)$-regular.
By Lemma~\ref{lem:NoKk-1}, $M$ is nonempty, so $H_\varphi$ is a complete graph on at most $k-2$ vertices; in particular, $H_\varphi$ is not $(k-2)$-regular.

If $B$ is not isomorphic to $K_{k-1}$, then by Lemma~\ref{lem:clique-cover}, every component of $H[L_\varphi(C)]$ is a $(k-2)$-regular $q$-blowup of a cycle or a complete graph
for some $q$. 
More precisely, writing $t$ for the regularity of $\wt{B}$, $q = \frac{k-2}{t}$.
Therefore, $H_\varphi$ is a subgraph of $\wt{B}$, and hence each vertex of $H_\varphi$ has degree at most $t < k-2$.
In particular, $H_\varphi$ is not $(k-2)$-regular.

As $H_\varphi$ is a component of $H[L_\varphi(C)]$, which is not $(k-2)$-regular,
 \begin{equation}\label{unique3}
 \mbox{  $\forall v_i\in C-M$, the color $\varphi(v_i')\in L(v'_i)$ is adjacent to the unique color 
 $\alpha_{v_i}\in L(v_i) \cap H_\varphi$. }
 \end{equation}

If $M=C$, then
$$\rho_h(M)=-m(B)+b((k-2)\lambda -1)-2\lambda \frac{b(k-2)}{2}
=-m(B)+b\big((k-2)\lambda -1-(k-2)\lambda \big)=-m(B)-b\leq -4,
$$
contradicting Observation~\ref{obs:geq-1}. So $M\neq C$.

Fix a vertex $v_i\in C-M$, let $G_i=G-C$, and let $L_i$ differ from $L\vert_{V(G)-C}$ only in that $\varphi(v_i')$ is deleted from $L(v'_i)$. Correspondingly, $h_i$ differs from $h$ only in that $h_i(v'_i)=h(v'_i)-1$.
 By~\eqref{unique3}, $G_i$ is not $(H[L_i(V(G_i))], L_i)$-colorable. So, there is an $h_i$-minimal subgraph $G'_i$ of $G_i$, and it contains $v'_i$. 
 Since $h_i(v'_i)\leq k-2$, $G'_i\neq K_{k-1}$, and $\rho_{G_i,h_i}(V(G'_i))\leq -2$.
 Then   
 $\rho_{G,h}(V(G'_i))\leq - 2 + (\lambda + 2) =\lambda .$
So, by Lemma~\ref{lem:j(k-2)}, there is a unique edge joining  $V(G'_i)$ with $V(G) \setminus V(G'_i)$.
Since $C\cap V(G'_i)=\emptyset$, this edge is $v_iv'_i$. Hence each edge connecting $C$ with $V(G)-C$ is a cut edge. 
As $B^*$ is $2$-edge-connected, it follows that $x_0^* \in C$,  contradicting
Lemma~\ref{notx0}. 
\end{proof}

\section{Discharging}\label{sec:discharging}
We will use the notions of  $\sigma_h(F)$, $\wt{d}(v)$, and $m(F)$ from Lemma~\ref{GDP'}. 
By Corollary~\ref{cor:h-geq3}, each vertex $v \in V(G)$ satisfies $h(v) \geq 3$.

We show that $\rho_h(G) \leq -2$, proving that $G$ in fact is not a counterexample to Theorem~\ref{thm:stronger}.
We use the following discharging procedure.
We write $\alpha = \frac{k-2}{2k-7}$. 

\begin{enumerate}
    \item  For each $v \in V(G)$, we give initial charge $\rho_{h}(v)$ to $v$. 
    For each pair $uv$, where $u,v\in V(G)$ are joined by $s \geq 1$ edges, we give initial charge 
    $-s(2\lambda +1) + 1$ to the pair $uv$. For each 
    pair $uv$ of non-adjacent  $u,v \in V(G)$, the initial charge of $uv$ is $0$.
    \item For each pair $uv$ of adjacent vertices in $G$, if $ s\geq 1$ edges connect $u$ with $v$, 
    the pair $uv$ receives charge
    $(s(2\lambda +1)-1)/2$ from each of $u$ and $v$. 
    \item Each non-low vertex $u \in S^*_0$ takes charge $\alpha$ along each edge $e$ that joins $u$ to a low vertex $v\in S^*_0$.
\end{enumerate}

Now, we consider the final charges of the vertices in $G$. 
For each $v \in  V(G)$, we write $ch^*(v)$ for the final charge of $v$.
We observe that the total charge in $G$ is $\rho_h(G)$. Additionally, the final charge of each vertex pair is $0$, so $\rho_h(G) =\sum_{v\in V(G)}ch^*(v)$.
Finally,
if $S^*_0 \neq V(G)$, then the total charge in $G[\overline {S^*_0}]$ is at most 
$\rho_{G,h}(\overline {S^*_0}) - \lambda  \leq 0$ by Lemma~\ref{lem:k-1k}.
Therefore,
\begin{equation}
\label{eqn:ch-sum}
\rho_h(G) = \sum_{v \in V(G)} ch^*(v) \leq \sum_{v \in S^*_0} ch^*(v).
\end{equation}

For each $v \in S^*_0$, if $ d(v) \geq  1+h(v)$, then
we consider several cases.
\begin{enumerate}
    \item[(N1)] \label{item:small-nonlow} 
    If $d(v) \leq k-1$, then $h(v) \leq k-2$, so 
    \[ch^*(v) \leq (h(v)\lambda  - 1) - d(v) \lambda  + \alpha d(v) = \lambda (h(v) - d(v)) - 1 + \alpha d(v) \leq -\lambda -1 + \alpha d(v) < -1.\]
    \item[(N2)] If $d(v) = k$, then
    \[ch^*(v) \leq ((k-1)\lambda +1) - k\lambda  + \alpha k = -\lambda +1+\alpha k \leq 0.\]

    \item[(N3)] If $d(v) \geq  k+1$, then 
    \begin{eqnarray*}
    ch^*(v) &\leq&  ((k-1)\lambda +1) - d(v) \lambda  + \alpha d(v) \\ &\leq &
    ((k-1)\lambda +1) - (k+1) \lambda  + \alpha (k+1)
    = 
    (-\lambda +1+\alpha k) - \lambda  + \alpha
    < -1. 
    \end{eqnarray*}
\end{enumerate}

For each $v \in S^*_0$, if $d(v)=h(v)=j$ for some fixed $j \in [3,\ldots, k-1]$, then we consider two cases:

\begin{enumerate}
\item[(L1)] If $j \leq k-2$, then
\begin{equation*}
\label{eqn:phi-neg}
ch^*(v)\leq (j\lambda -1)-j\lambda -\frac{d(v)-\wt d(v)}{2}- \alpha ( d_{B^*}(v)-d_{{\BB}_0}(v))
= -1-\frac{d(v)-\wt d(v)}{2}- \alpha ( d_{B^*}(v)-d_{{\BB}_0}(v)).
\end{equation*}
\item[(L2)]  If $j=k-1$, then 
\begin{eqnarray*}
\notag
ch^*(v) &\leq& (k-1)\lambda +1-(k-1)\lambda -\frac{d(v)-\wt d(v)}{2}-\alpha(d_{B^*}(v)-d_{{\BB}_0}(v)) \\
\label{eqn:phi-pos}
&=& 1-\frac{d(v)-\wt d(v)}{2}-\alpha d_{B^*}(v)+ 
\alpha d_{{\BB}_0}(v).
\end{eqnarray*}
\end{enumerate}

We claim that ${\BB}_0$ is nonempty. Indeed, suppose that ${\BB}_0$ is empty.
As $d(v) \geq 1+h(v)$ for each $v \in S_0^*$ we have the following inequality instead of  (N1)--(N3): 
\[ch^*(v) \leq ( h(v)\lambda  + 1) - d(v) \lambda  = (h(v) - d(v))\lambda  + 1 \leq -\lambda +1 < -2.\]
Then, by (\ref{eqn:ch-sum}), 
$\rho_h(G) < -2 |S^*_0| \leq   -2$, so $G$ is not a counterexample.
Thus, we assume that ${\BB}_0 \neq \emptyset$, or equivalently, that $S_0^*$ contains at least one low vertex.

\begin{lemma}\label{charg2}
If $B$ is a component of ${\BB}_0$,
then $\sum_{v\in V(B)}ch^*(v) < -1$.
\end{lemma}
\begin{proof} We note that  $d_{B^*}(v)=h(v)$ for $v\in V(B) \setminus \{x^*_0 \}$, and $d_{B^*}(x^*_0) = h(x^*_0) - 1$ whenever $x^*_0 \in V({\BB}_0)$.
Hence, by (L1)--(L2),
$$\sum_{v\in V(B)}ch^*(v)\leq |V_{k-1}(B)|-|V_{k-1}^-(B)|-\sum_{v\in V(B)}\left(\frac{d(v)-\wt d(v)}{2} \right ) - \alpha \sum_{v \in V(B)} \left (h(v)-d_{{\BB}_0}(v)
\right)+ \alpha $$
$$
= |V_{k-1}(B)|-|V_{k-1}^-(B)|-m(B)- \alpha \sigma_h(B) + \alpha =-\Phi_k(B) + \alpha.$$
Note that the $\alpha$ term accounts for the possibility that $x^*_0 \in V(B)$, in which case $-d_{B^*}(x^*_0) = -h(x^*_0) + 1$.

By Corollary~\ref{cor:h-geq3}, Condition (i) of 
Lemma~\ref{GDP'} holds for $T=B$. Condition (ii) holds because the vertices in $B$ are low.
By Lemma~\ref{Nothing}, 
$B$ has no $(k-2)$-regular block. 
If $B$ is $(k-1)$-regular, then as $B$ is low and $G$ is connected, $G = B$, which contradicts Lemma~\ref{lem:k-1-reg}.
Therefore, Condition (iii) holds.
 So, by Lemma~\ref{GDP'},
$\sum_{v\in V(B)}ch^*(v)\leq -\Phi_k(B) + \alpha < -1$.
\end{proof}

Now, by (N1)--(N3), the vertices of $S_0^* \setminus {\BB}_0$ have total nonpositive charge. 
Therefore,
\[\rho_h(G) =  \sum_{v \in V(G)} ch^*(v) \leq \sum_{v \in S^*_0} ch^*(v) \leq \sum_{v \in {\BB}_0} ch^*(v).\]
As ${\BB}_0 \neq \emptyset$, Lemma~\ref{charg2} implies that 
$\sum_{v \in {\BB}_0} ch^*(v) < -1$. 
Therefore, $\rho_h(G) < -1$, and as $\rho_h(G)$ is integral, $\rho_h(G) \leq -2$.
This completes the proof.

\section{Bound for list coloring}\label{listco}

We say that $G$ is \emph{$(H,L)$-minimal} if $G$ has no $(H,L)$-coloring but every proper subgraph $G'$ of $G$ has an $(H[L(V(G'))], L\vert_{V(G')})$-coloring.

\begin{lemma}
\label{lem:HLcrit}
    Let $k \geq 5$, and let $G$ be a multigraph
    with a function $h:V(G) \to [0,k-1]$. Suppose that $G$ has no exceptional subgraph. 
    If $(H,L)$ is an $h$-cover of $G$ and $G$ is $(H,L)$-minimal, then $\rho_h(G) \leq -2$.
\end{lemma}
\begin{proof}
    Suppose for the sake of contradiction that $\rho_h(G) \geq -1$.
    As $G$ is $(H,L)$-minimal, no vertex $v \in V(G)$ satisfies $h(v) = 0$.
    By Theorem~\ref{thm:stronger}, $G$ is not $h$-minimal.
    Since $G$ is not DP $h$-colorable, $G$ has a proper subgraph $G_0$ that is $h$-minimal. 
    Since $G$ has no exceptional subgraph, Theorem~\ref{thm:stronger} implies that $\rho_{G,h}(V(G_0)) \leq -2$.
    We let $G_1 \supseteq G_0$ be a maximal subgraph of $G$ for which $\rho_{G,h}(V(G_1)) \leq -2$, and we observe that $V(G_1) \subsetneq V(G)$.
    By the $(H,L)$-minimality of $G$, $G_1$ has an $(H,L)$-coloring $\varphi$.

    Now, define $G' = G - V(G_1)$.
    We observe that for each $v \in V(G')$, if $|E_G(v,V(G_1))| = j$, then $h(v) \geq j+1$. 
    Indeed, if $j=0$, then the inequality is clear. If $j \geq 1$ and 
    $h(v) \leq j$, then $\rho_{G,h}(V(G_1) \cup \{v\}) \leq -2 + j \lambda + 1 - 2 j \lambda < -2$, contradicting the maximality of $G_1$.
    
    For each vertex $v \in V(G')$, write $h'(v) = h(v) - |E_G(v,V(G_1))|$. By our previous observation, $h'(v) \geq 1$ for each $v \in V(G')$.
    Consider a vertex subset $U \subseteq V(G')$, and let $j = |E_G(U,V(G_1))|$. By the maximality of $G_1$, 
    \[-1 \leq \rho_{G,h}(V(G_1) \cup U) \leq -2 - 2j\lambda + \rho_{G,h}(U).\]
    Rearranging, $\rho_{G,h}(U) \geq 1 + 2j\lambda$, which implies that $\rho_{G,h'}(U) \geq 1 + j(\lambda - 2)$.
    Thus,
    every subgraph $G''$ of $G'$ satisfies $\rho_{h'}(G'') \geq 1$, and hence by Theorem~\ref{thm:stronger}, $G'$ has no $h'$-critical subgraph. Therefore, $G'$ is DP $h'$-colorable, and thus 
    $\varphi$ can be extended to an $(H,L)$-coloring of $G$. This contradicts the original assumption that $G$ is $(H,L)$-minimal, completing the proof.
\end{proof}

Now, with Lemma~\ref{lem:HLcrit}, we can prove Theorem~\ref{thm:fl_intro}.


\begin{thm:fl_intro}
    If $k \geq 5$,  $\lambda  = \left \lceil \frac{k^2 - 7}{2k-7}\right \rceil$, and $n \geq k+2$, then $f_{\ell}(n,k) \geq \left(k-1 + \frac {1}{\lambda} \right) \frac n2 + \frac 1 {\lambda}$.
\end{thm:fl_intro}
\begin{proof}[Proof of Theorem~\ref{thm:fl_intro}]
    Let $G$ be a list $k$-critical multigraph on $n \geq k+2$ vertices and $m$ edges. As $G$ has no $k$-critical proper subgraph, $G$ is $K_k$-free.
    Furthermore, as $G$ is list $k$-critical, $G$ has no parallel edges. Therefore, $G$ has no exceptional subgraph.
    Let $L': V(G) \to \powerset{\N}$ be a list assignment for which $|L'(v)|\geq k-1$ for every $v\in V(G)$ and $G$ has no $L'$-coloring. 

    Now, we construct a $(k-1)$-cover $(H,L)$ of $G$ for which the $L'$-coloring and the $(H,L)$-coloring problems on $G$ are equivalent. As $G$ has no $L'$-coloring and is list $k$-critical, it follows that $G$ is $(H,L)$-minimal. 
    Therefore, by Lemma~\ref{lem:HLcrit}, $\rho_h(G) = (k \lambda +1 )n  -2 \lambda m \leq -2$. Rearranging, $m \geq  (k-1 + \frac {1}{\lambda} ) \frac n2 + \frac 1 {\lambda}$. 
\end{proof}

{\bf Concluding remarks.} We do not think that our bounds are sharp. 
On the other hand, it seems that $\fDP(n,k)<f(n,k)$ for many values of $n$ and $k$. 
An interesting problem is to find the asymptotics of $\frac{\fDP(n,k)}{n}$ for a fixed $k$ and $n\to\infty$.

\bibliographystyle{abbrv}
{\small
\bibliography{ref}}

@article {2018KoYa,
    AUTHOR = {Kostochka, Alexandr and Yancey, Matthew},
     TITLE = {A {B}rooks-type result for sparse critical graphs},
   JOURNAL = {Combinatorica},
  FJOURNAL = {Combinatorica. An International Journal on Combinatorics and
              the Theory of Computing},
    VOLUME = {38},
      YEAR = {2018},
    NUMBER = {4},
     PAGES = {887--934},
      ISSN = {0209-9683,1439-6912},
   MRCLASS = {05C15 (05C35)},
  MRNUMBER = {3850010},
MRREVIEWER = {Gregory\ J.\ Puleo},
       DOI = {10.1007/s00493-017-3068-3},
       URL = {https://doi.org/10.1007/s00493-017-3068-3},
}

@article {1963Gallai1,
    AUTHOR = {Gallai, T.},
     TITLE = {Kritische {G}raphen. {I}},
   JOURNAL = {Magyar Tud. Akad. Mat. Kutat\'{o} Int. K\"{o}zl.},
  FJOURNAL = {A Magyar Tudom\'{a}nyos Akad\'{e}mia. Matematikai Kutat\'{o}
              Int\'{e}zet\'{e}nek K\"{o}zlem\'{e}nyei},
    VOLUME = {8},
      YEAR = {1963},
     PAGES = {165--192},
      ISSN = {0541-9514},
   MRCLASS = {05.55},
  MRNUMBER = {188099},
MRREVIEWER = {R.\ C.\ Read},
}

@article {1963Gallai2,
    AUTHOR = {Gallai, T.},
     TITLE = {Kritische {G}raphen. {II}},
   JOURNAL = {Magyar Tud. Akad. Mat. Kutat\'{o} Int. K\"{o}zl.},
  FJOURNAL = {A Magyar Tudom\'{a}nyos Akad\'{e}mia. Matematikai Kutat\'{o}
              Int\'{e}zet\'{e}nek K\"{o}zlem\'{e}nyei},
    VOLUME = {8},
      YEAR = {1963},
     PAGES = {373--395},
      ISSN = {0541-9514},
   MRCLASS = {05.55},
  MRNUMBER = {188100},
MRREVIEWER = {R.\ C.\ Read},
}

@article {k=4,
    AUTHOR = {Peter Bradshaw and Ilkyoo Choi and Alexandr Kostochka and Jingwei Xu},
     TITLE = {A lower bound on the number of edges in {DP}-critical graphs. {II.} {Four} colors},
   NOTE = {In preparation},
    YEAR = {2024},
}

@article {1997Krivelevich,
    AUTHOR = {Krivelevich, Michael},
     TITLE = {On the minimal number of edges in color-critical graphs},
   JOURNAL = {Combinatorica},
  FJOURNAL = {Combinatorica. An International Journal on Combinatorics and
              the Theory of Computing},
    VOLUME = {17},
      YEAR = {1997},
    NUMBER = {3},
     PAGES = {401--426},
      ISSN = {0209-9683,1439-6912},
   MRCLASS = {05C15 (05C35)},
  MRNUMBER = {1606048},
MRREVIEWER = {W.\ G.\ Brown},
       DOI = {10.1007/BF01215921},
       URL = {https://doi.org/10.1007/BF01215921},
}

@article{Hea,
title = {Map colour theorem},
journal = {J. Pure Appl. Math.},
volume = {24},
pages = {332-338},
year = {1890},
author = {P. J. Heawood},
}

@article {KSW,
    AUTHOR = {Kostochka, A. V. and Stiebitz, M. and Wirth, B.},
     TITLE = {The colour theorems of {B}rooks and {G}allai extended},
   JOURNAL = {Discrete Math.},
  FJOURNAL = {Discrete Mathematics},
    VOLUME = {162},
      YEAR = {1996},
    NUMBER = {1-3},
     PAGES = {299--303},
      ISSN = {0012-365X,1872-681X},
   MRCLASS = {05C15},
  MRNUMBER = {1425799},
       DOI = {10.1016/0012-365X(95)00294-7},
       URL = {https://doi-org.yproxy.hufs.ac.kr/10.1016/0012-365X(95)00294-7},
}

@article {Gro,
    AUTHOR = {Gr\"otzsch, Herbert},
     TITLE = {Zur {T}heorie der diskreten {G}ebilde. {VII}. {E}in
              {D}reifarbensatz f\"ur dreikreisfreie {N}etze auf der {K}ugel},
   JOURNAL = {Wiss. Z. Martin-Luther-Univ. Halle-Wittenberg Math.-Natur.
              Reihe},
  FJOURNAL = {Wissenschaftliche Zeitschrift der Martin-Luther-Universit\"at
              Halle-Wittenberg. Mathematisch-Naturwissenschaftliche Reihe},
    VOLUME = {8},
      YEAR = {1958/59},
     PAGES = {109--120},
      ISSN = {0138-1504},
   MRCLASS = {55.00},
  MRNUMBER = {116320},
}

@article {Bro,
    AUTHOR = {Brooks, R. L.},
     TITLE = {On colouring the nodes of a network},
   JOURNAL = {Proc. Cambridge Philos. Soc.},
  FJOURNAL = {Proceedings of the Cambridge Philosophical Society},
    VOLUME = {37},
      YEAR = {1941},
     PAGES = {194--197},
      ISSN = {0008-1981},
   MRCLASS = {56.0X},
  MRNUMBER = {12236},
MRREVIEWER = {P.\ Franklin},
}

@article {Tho,
    AUTHOR = {Thomassen, Carsten},
     TITLE = {Color-critical graphs on a fixed surface},
   JOURNAL = {J. Combin. Theory Ser. B},
  FJOURNAL = {Journal of Combinatorial Theory. Series B},
    VOLUME = {70},
      YEAR = {1997},
    NUMBER = {1},
     PAGES = {67--100},
      ISSN = {0095-8956,1096-0902},
   MRCLASS = {05C15 (57M15)},
  MRNUMBER = {1441260},
MRREVIEWER = {Lorenzo\ Traldi},
       DOI = {10.1006/jctb.1996.1722},
       URL = {https://doi-org.yproxy.hufs.ac.kr/10.1006/jctb.1996.1722},
}

@article {KimO,
    AUTHOR = {Kim, Seog-Jin and Ozeki, Kenta},
     TITLE = {A note on a {B}rooks' type theorem for {DP}-coloring},
   JOURNAL = {J. Graph Theory},
  FJOURNAL = {Journal of Graph Theory},
    VOLUME = {91},
      YEAR = {2019},
    NUMBER = {2},
     PAGES = {148--161},
      ISSN = {0364-9024,1097-0118},
   MRCLASS = {05C15},
  MRNUMBER = {3948125},
MRREVIEWER = {Bin\ Liu},
       DOI = {10.1002/jgt.22425},
       URL = {https://doi-org.yproxy.hufs.ac.kr/10.1002/jgt.22425},
}

@book {2024StScTo,
    AUTHOR = {Stiebitz, Michael and Schweser, Thomas and Toft, Bjarne},
     TITLE = {Brooks' theorem---graph coloring and critical graphs},
    SERIES = {Springer Monographs in Mathematics},
 PUBLISHER = {Springer, Cham},
      YEAR = {[2024] \copyright 2024},
     PAGES = {xiv+655},
      ISBN = {978-3-031-50064-0; 978-3-031-50065-7},
   MRCLASS = {05C15 (05-02 05C10 05C35 05C65)},
  MRNUMBER = {4755316},
       DOI = {10.1007/978-3-031-50065-7},
       URL = {https://doi-org.yproxy.hufs.ac.kr/10.1007/978-3-031-50065-7},
}

@book {Cran,
    AUTHOR = {Cranston, Daniel},
     TITLE = {Graph Coloring Methods},
 PUBLISHER = {Daniel W. Cranston, Richmond, Virginia},
      YEAR = {2024},
     PAGES = {viii+453},
      ISBN = {979-8-218-46242-0 (paperback)},
   MRCLASS = {05C15},
  MRNUMBER = {4824381},
       URL = {https://graphcoloringmethods.com},
}

@article {KS02,
    AUTHOR = {Kostochka, Alexandr V. and Stiebitz, Michael},
     TITLE = {A list version of {D}irac's theorem on the number of edges in
              colour-critical graphs},
   JOURNAL = {J. Graph Theory},
  FJOURNAL = {Journal of Graph Theory},
    VOLUME = {39},
      YEAR = {2002},
    NUMBER = {3},
     PAGES = {165--177},
      ISSN = {0364-9024,1097-0118},
   MRCLASS = {05C15},
  MRNUMBER = {1883593},
       DOI = {10.1002/jgt.998},
       URL = {https://doi-org.proxy2.library.illinois.edu/10.1002/jgt.998},
}

@article {1998Krivelevich,
    AUTHOR = {Krivelevich, Michael},
     TITLE = {An improved bound on the minimal number of edges in
              color-critical graphs},
   JOURNAL = {Electron. J. Combin.},
  FJOURNAL = {Electronic Journal of Combinatorics},
    VOLUME = {5},
      YEAR = {1998},
     PAGES = {Research Paper 4, 4},
      ISSN = {1077-8926},
   MRCLASS = {05C15 (05C35)},
  MRNUMBER = {1486397},
       DOI = {10.37236/1342},
       URL = {https://doi.org/10.37236/1342},
}

@article {2014KoYa,
    AUTHOR = {Kostochka, Alexandr and Yancey, Matthew},
     TITLE = {Ore's conjecture on color-critical graphs is almost true},
   JOURNAL = {J. Combin. Theory Ser. B},
  FJOURNAL = {Journal of Combinatorial Theory. Series B},
    VOLUME = {109},
      YEAR = {2014},
     PAGES = {73--101},
      ISSN = {0095-8956,1096-0902},
   MRCLASS = {05C15 (05C85)},
  MRNUMBER = {3269903},
MRREVIEWER = {Andr\'{a}s\ Gy\'{a}rf\'{a}s},
       DOI = {10.1016/j.jctb.2014.05.002},
       URL = {https://doi.org/10.1016/j.jctb.2014.05.002},
}

@article {1976Vizing,
    AUTHOR = {Vizing, V. G.},
     TITLE = {Coloring the vertices of a graph in prescribed colors},
   JOURNAL = {Diskret. Analiz},
  FJOURNAL = {Akademiya Nauk SSSR. Sibirskoe Otdelenie. Institut Matematiki.
              Diskretny\u{i} Analiz. Sbornik Trudov},
      YEAR = {1976},
     VOLUME = {29},
     PAGES = {3--10, 101},
   MRCLASS = {05C15},
  MRNUMBER = {498216},
MRREVIEWER = {D.\ Cvetkovi\'{c}},
}

@inproceedings {1980ErRuTa,
    AUTHOR = {Erd\H{o}s, Paul and Rubin, Arthur L. and Taylor, Herbert},
     TITLE = {Choosability in graphs},
 BOOKTITLE = {Proceedings of the {W}est {C}oast {C}onference on
              {C}ombinatorics, {G}raph {T}heory and {C}omputing ({H}umboldt
              {S}tate {U}niv., {A}rcata, {C}alif., 1979)},
    SERIES = {Congress. Numer.},
    VOLUME = {XXVI},
     PAGES = {125--157},
 PUBLISHER = {Utilitas Math., Winnipeg, MB},
      YEAR = {1980},
      ISBN = {0-919628-26-5},
   MRCLASS = {05C15},
  MRNUMBER = {593902},
MRREVIEWER = {R.\ C.\ Entringer},
}

@article {2018DvPo,
    AUTHOR = {Dvo\v{r}\'{a}k, Zden\v{e}k and Postle, Luke},
     TITLE = {Correspondence coloring and its application to list-coloring
              planar graphs without cycles of lengths 4 to 8},
   JOURNAL = {J. Combin. Theory Ser. B},
  FJOURNAL = {Journal of Combinatorial Theory. Series B},
    VOLUME = {129},
      YEAR = {2018},
     PAGES = {38--54},
      ISSN = {0095-8956,1096-0902},
   MRCLASS = {05C10 (05C15)},
  MRNUMBER = {3758240},
MRREVIEWER = {Deming\ Li},
       DOI = {10.1016/j.jctb.2017.09.001},
       URL = {https://doi.org/10.1016/j.jctb.2017.09.001},
}

@article {2017BeKoPr,
    AUTHOR = {Bernshteyn, A. Yu. and Kostochka, A. V. and Pron’, S. P.},
     TITLE = {On {DP}-coloring of graphs and multigraphs},
   JOURNAL = {Sibirsk. Mat. Zh.},
  FJOURNAL = {Rossi\u{i}skaya Akademiya Nauk. Sibirskoe Otdelenie. Institut
              Matematiki im. S. L. Soboleva. Sibirski\u{i} Matematicheski\u{i} Zhurnal},
    VOLUME = {58},
      YEAR = {2017},
    NUMBER = {1},
     PAGES = {36--47},
      ISSN = {0037-4474},
   MRCLASS = {05C15 (05C69 05C75)},
  MRNUMBER = {3686937},
MRREVIEWER = {Zden\v{e}k\ Ryj\'{a}\v{c}ek},
       DOI = {10.1134/s0037446617010049},
       URL = {https://doi.org/10.1134/s0037446617010049},
}

@article {1951Dirac,
    AUTHOR = {Dirac, G. A.},
     TITLE = {Note on the colouring of graphs},
   JOURNAL = {Math. Z.},
  FJOURNAL = {Mathematische Zeitschrift},
    VOLUME = {54},
      YEAR = {1951},
     PAGES = {347--353},
      ISSN = {0025-5874,1432-1823},
   MRCLASS = {56.0X},
  MRNUMBER = {46030},
MRREVIEWER = {W.\ T.\ Tutte},
       DOI = {10.1007/BF01238034},
       URL = {https://doi.org/10.1007/BF01238034},
}

@article {1952DiracSome,
    AUTHOR = {Dirac, G. A.},
     TITLE = {Some theorems on abstract graphs},
   JOURNAL = {Proc. London Math. Soc. (3)},
  FJOURNAL = {Proceedings of the London Mathematical Society. Third Series},
    VOLUME = {2},
      YEAR = {1952},
     PAGES = {69--81},
      ISSN = {0024-6115,1460-244X},
   MRCLASS = {56.0X},
  MRNUMBER = {47308},
MRREVIEWER = {W.\ T.\ Tutte},
       DOI = {10.1112/plms/s3-2.1.69},
       URL = {https://doi.org/10.1112/plms/s3-2.1.69},
}

@article {1957DiracAtheorem,
    AUTHOR = {Dirac, G. A.},
     TITLE = {A theorem of {R}. {L}. {B}rooks and a conjecture of {H}.
              {H}adwiger},
   JOURNAL = {Proc. London Math. Soc. (3)},
  FJOURNAL = {Proceedings of the London Mathematical Society. Third Series},
    VOLUME = {7},
      YEAR = {1957},
     PAGES = {161--195},
      ISSN = {0024-6115,1460-244X},
   MRCLASS = {55.0X},
  MRNUMBER = {86305},
MRREVIEWER = {W.\ T.\ Tutte},
       DOI = {10.1112/plms/s3-7.1.161},
       URL = {https://doi.org/10.1112/plms/s3-7.1.161},
}

@article {1953Dirac,
    AUTHOR = {Dirac, G. A.},
     TITLE = {The structure of {$k$}-chromatic graphs},
   JOURNAL = {Fund. Math.},
  FJOURNAL = {Polska Akademia Nauk. Fundamenta Mathematicae},
    VOLUME = {40},
      YEAR = {1953},
     PAGES = {42--55},
      ISSN = {0016-2736,1730-6329},
   MRCLASS = {56.0X},
  MRNUMBER = {60207},
MRREVIEWER = {W.\ T.\ Tutte},
       DOI = {10.4064/fm-40-1-42-55},
       URL = {https://doi.org/10.4064/fm-40-1-42-55},
}

@article {1974Dirac,
    AUTHOR = {Dirac, G. A.},
     TITLE = {The number of edges in critical graphs},
   JOURNAL = {J. Reine Angew. Math.},
  FJOURNAL = {Journal f\"{u}r die Reine und Angewandte Mathematik. [Crelle's
              Journal]},
    VOLUME = {268/269},
      YEAR = {1974},
     PAGES = {150--164},
      ISSN = {0075-4102,1435-5345},
   MRCLASS = {05C15},
  MRNUMBER = {345858},
MRREVIEWER = {Torrence\ D.\ Parsons},
       DOI = {10.1515/crll.1974.268-269.150},
       URL = {https://doi.org/10.1515/crll.1974.268-269.150},
}

@article {1957DiracMap,
    AUTHOR = {Dirac, G. A.},
     TITLE = {Map colour theorems related to the {H}eawood colour formula.
              {II}},
   JOURNAL = {J. London Math. Soc.},
  FJOURNAL = {The Journal of the London Mathematical Society},
    VOLUME = {32},
      YEAR = {1957},
     PAGES = {436--455},
      ISSN = {0024-6107,1469-7750},
   MRCLASS = {55.0X},
  MRNUMBER = {89413},
MRREVIEWER = {W.\ T.\ Tutte},
       DOI = {10.1112/jlms/s1-32.4.436},
       URL = {https://doi.org/10.1112/jlms/s1-32.4.436},
}

@article {2018CrRa,
    AUTHOR = {Cranston, Daniel W. and Rabern, Landon},
     TITLE = {Edge lower bounds for list critical graphs, via discharging},
   JOURNAL = {Combinatorica},
  FJOURNAL = {Combinatorica. An International Journal on Combinatorics and
              the Theory of Computing},
    VOLUME = {38},
      YEAR = {2018},
    NUMBER = {5},
     PAGES = {1045--1065},
      ISSN = {0209-9683,1439-6912},
   MRCLASS = {05C15 (05C35)},
  MRNUMBER = {3884777},
MRREVIEWER = {Gregory\ J.\ Puleo},
       DOI = {10.1007/s00493-016-3584-6},
       URL = {https://doi.org/10.1007/s00493-016-3584-6},
}

@book {1979Borodin,
    AUTHOR = {Borodin, O. V.},
     TITLE = {Problems of coloring and of covering the vertex set of a graph by induced subgraphs (in Russian)},
      NOTE = {Thesis (Ph.D.)--Novosibirsk State University},
 PUBLISHER = {Novosibirsk},
      YEAR = {1979},
     PAGES = {},
      ISBN = {},
   MRCLASS = {},
  MRNUMBER = {},
       URL = {},
}

@article{2003KS,
title = {A new lower bound on the number of edges in colour-critical graphs and hypergraphs},
journal = {Journal of Combinatorial Theory, Series B},
volume = {87},
number = {2},
pages = {374-402},
year = {2003},
issn = {0095-8956},
doi = {https://doi.org/10.1016/S0095-8956(02)00035-7},
url = {https://www.sciencedirect.com/science/article/pii/S0095895602000357},
author = {Alexandr V. Kostochka and Michael Stiebitz}
}

@article{2020KiRa,
title = {Improved lower bounds on the number of edges in list critical and online list critical graphs},
journal = {Journal of Combinatorial Theory, Series B},
volume = {140},
pages = {147-170},
year = {2020},
issn = {0095-8956},
doi = {https://doi.org/10.1016/j.jctb.2019.05.004},
url = {https://www.sciencedirect.com/science/article/pii/S009589561930053X},
author = {H.A. Kierstead and Landon Rabern}
}

@article {2016Rabern,
    AUTHOR = {Rabern, Landon},
     TITLE = {A better lower bound on average degree of 4-list-critical
              graphs},
   JOURNAL = {Electron. J. Combin.},
  FJOURNAL = {Electronic Journal of Combinatorics},
    VOLUME = {23},
      YEAR = {2016},
    NUMBER = {3},
     PAGES = {Paper 3.37, 5},
      ISSN = {1077-8926},
   MRCLASS = {05C15},
  MRNUMBER = {3558074},
MRREVIEWER = {Xuechao\ Li},
       DOI = {10.37236/5971},
       URL = {https://doi.org/10.37236/5971},
}

@article {2018Rabern,
    AUTHOR = {Rabern, Landon},
     TITLE = {A better lower bound on the average degree of online
              {$k$}-list-critical graphs},
   JOURNAL = {Electron. J. Combin.},
  FJOURNAL = {Electronic Journal of Combinatorics},
    VOLUME = {25},
      YEAR = {2018},
    NUMBER = {1},
     PAGES = {Paper No. 1.51, 15},
      ISSN = {1077-8926},
   MRCLASS = {05C15 (05C07 05C35)},
  MRNUMBER = {3785030},
MRREVIEWER = {Owen\ D.\ Byer},
       DOI = {10.37236/6405},
       URL = {https://doi.org/10.37236/6405},
}

@article {2018BeKo,
    AUTHOR = {Bernshteyn, Anton and Kostochka, Alexandr},
     TITLE = {Sharp {D}irac's theorem for {DP}-critical graphs},
   JOURNAL = {J. Graph Theory},
  FJOURNAL = {Journal of Graph Theory},
    VOLUME = {88},
      YEAR = {2018},
    NUMBER = {3},
     PAGES = {521--546},
      ISSN = {0364-9024,1097-0118},
   MRCLASS = {05C15 (05C35)},
  MRNUMBER = {3803061},
MRREVIEWER = {Nedyalko\ Nenov},
       DOI = {10.1002/jgt.22227},
       URL = {https://doi.org/10.1002/jgt.22227},
}
 
  \end{document}